\documentclass[svgnames]{amsart}
\usepackage[english]{babel}
\usepackage[utf8]{inputenc}

\usepackage[margin=3cm]{geometry}

\usepackage[colorlinks=true, pdfstartview=FitV, linkcolor=black, citecolor=black, urlcolor=black]{hyperref}
\usepackage[nameinlink]{cleveref}

\usepackage{amssymb, amsfonts, amsmath, amsthm, amscd, amsxtra, latexsym, mathabx, mathtools, enumitem}
\usepackage{tikz}\usetikzlibrary{cd}
\usepackage{epigraph}

\theoremstyle{plain}
\newtheorem{thm}{Theorem}[section] 
\newtheorem{lem}[thm]{Lemma}
\newtheorem{cor}[thm]{Corollary}
\newtheorem{prop}[thm]{Proposition}
\newtheorem{thmintro}{Theorem}

\newtheorem{corintro}[thmintro]{Corollary}
\newtheorem{claim}[thm]{Claim}
\newenvironment{claimproof}{\begin{proof}}{\end{proof}}

\theoremstyle{definition}
\newtheorem{defn}[thm]{Definition}
\newtheorem{rem}[thm]{Remark}
\newtheorem{remintro}[thmintro]{Remark}

\newtheorem{notation}[thm]{Notation}

\newcounter{gcomments}

\newcounter{acomments}

\newcommand{\frakS}{\mathfrak S}
\newcommand{\C}[1]{\mathcal{C}(#1)}
\newcommand{\dist}{\mathrm{d}}
\newcommand{\diam}{\mathrm{diam}}
\newcommand{\propnest}{\sqsubsetneq}
\newcommand{\orth}{\bot}
\newcommand{\transverse}{\pitchfork}
\newcommand{\nest}{\sqsubseteq}

\newcommand{\link}[2]{\operatorname{Link}_{#1}\left(#2\right)}
\newcommand{\W}{\mathcal{W}}
\newcommand{\duaug}[2]{{#1}^{+{#2}}}
\newcommand{\Sat}{\operatorname{Sat}}
\newcommand{\CStar}[1]{\operatorname{Star}\left(#1\right)}

\newcommand{\p}{\mathfrak{p}}
\newcommand{\Cone}[1]{\operatorname{Cone}\left(#1\right)}
\newcommand{\Hor}[1]{\operatorname{Hor}\left(#1\right)}
\newcommand{\hatC}[1]{\hat{\mathcal{C}}(#1)}
\newcommand{\depth}[1]{\operatorname{Depth}\left(#1\right)}
\newcommand{\down}[1]{{#1}^{\downarrow}}
\newcommand{\squid}[1]{squid}
\newcommand{\Squid}[1]{Squid}

\newcommand{\N}{\mathbb{N}}
\newcommand{\Z}{\mathbb{Z}}
\newcommand{\R}{\mathbb{R}}

\newcommand{\MCG}{\mathcal{MCG}}
\newcommand{\Cay}[2]{\operatorname{Cay}\left(#1,#2\right)}
\newcommand{\Stab}[2]{\operatorname{Stab}_{#1}\left(#2\right)}
\newcommand{\inj}[2]{\operatorname{inj}_{#2}{#1}}
\newcommand{\Aut}[1]{\operatorname{Aut}\left(#1\right)}

\newcommand{\Teich}[1]{\operatorname{Teich}\left(#1\right)}

\renewcommand{\hat}{\widehat}

\newcommand{\ov}[1]{\overline{#1}}

\title[Cusped spaces for HHG]{Cusped spaces for hierarchically hyperbolic groups, and applications to Dehn filling quotients}

\author[G. Mangioni]{Giorgio Mangioni}
    \address{(Giorgio Mangioni) Maxwell Institute and Department of Mathematics, Heriot-Watt University, Edinburgh, UK}
    \email{gm2070@hw.ac.uk}

\author[A. Sisto]{Alessandro Sisto}
    \address{(Alessandro Sisto) Maxwell Institute and Department of Mathematics, Heriot-Watt University, Edinburgh, UK}
    \email{a.sisto@hw.ac.uk}

\begin{document}

\begin{abstract}
    We introduce a construction that simultaneously yields cusped spaces of relatively hyperbolic groups, and spaces quasi-isometric to Teichm\"{u}ller metrics. We use this to study Dehn-filling-like quotients of various groups, among which mapping class groups $\MCG(S_n)$ of $n$-punctured spheres. In particular, we show that $\MCG(S_5)$ (resp. the braid group on four strands) has infinite hyperbolic quotients (strongly) not isomorphic to hyperbolic quotients of any other given $\MCG(S_n)$ (resp. any other braid group). These quotients are obtained from $\MCG(S_5)$ by modding out suitable large powers of Dehn twists, and we further argue that the corresponding quotients of the extended mapping class group have trivial outer automorphism groups. We obtain these results by studying torsion elements in the relevant quotients.
\end{abstract}

\maketitle


\section*{Introduction}
 Cusped spaces, as introduced in \cite{GrovesManning, Bow-relhyp}, are fundamental tools in the study of relatively hyperbolic groups. For instance, relevantly to this paper, their interplay with algebraic Dehn fillings has produced countless applications; a non-exhaustive list of references includes \cite{GrovesManning,AGM_dehnfill,virtual_Haken,  dahmaniguirardel, GM-elementary, isomorphy, GMS,  quasiconvexity_and_Dehn_filling, Wang}. Cusped spaces share similarities with Teichm\"{u}ller space, regarded as a coarse geometric object when endowed with the Teichm\"{u}ller metric, in that its thin part, where the mapping class group action is not cocompact, can be described in terms of products of horoballs \cite{Minsky_productregion,Durham_augmented_markings}. We unify these perspectives by introducing a construction of cusped spaces for relatively hierarchically hyperbolic groups (HHGs), as defined in \cite{HHS_I, HHS_II}, that simultaneously yields cusped spaces and Teichm\"{u}ller spaces. 
 
 The starting point is a relatively HHG admitting a combinatorial model analogous to Masur-Minsky's marking complexes \cite{Masur_Minsky_2}, for instance a relatively hyperbolic group or a mapping class group. The precise requirements on the group, which we defer to \Cref{defn:rel_squid_HHG}, are not very restrictive in view of results from \cite{converse}, see discussion below. The construction of the cusped space is then essentially a generalisation of Durham's augmented marking complex for the Teichm\"{u}ller metric \cite{Durham_augmented_markings}, as we explain in more detail in \Cref{rem:augmented_markings}. For the reader familiar with hierarchical hyperbolicity, we roughly replace every hyperbolic space at the bottom of the hierarchy with the combinatorial horoball over it; this is described in more detail below, and implemented in \Cref{defn:cusped_space}. With this in mind, the first properties of our construction are:

\begin{thmintro}
\label{thmintro:first}
    Let $G$ be a relative HHG satisfying \Cref{defn:rel_squid_HHG}, and let $\mathcal Z$ be its cusped space as in \Cref{defn:cusped_space}. Then:
    \begin{enumerate}
        \item $\mathcal Z$ is a hierarchically hyperbolic space;
        \item $G$ acts on $\mathcal Z$ and the orbit maps are coarse embeddings;
        \item for $G$ relatively hyperbolic, $\mathcal Z$ is $G$-equivariantly quasi-isometric to the cusped space of $G$;
        \item for $G$ a mapping class group of a finite-type surface, $\mathcal Z$ is $G$-equivariantly quasi-isometric to the corresponding Teichm\"{u}ller metric. 
    \end{enumerate}
\end{thmintro}

\noindent  We expect our cusped spaces to serve as a general-purpose tool in the future. For instance, it might be possible to use them to prove the Farrell-Jones Conjecture for a large class of relative HHGs, combining techniques from \cite{relative_FJC} and \cite{durham2025asymptoticallycat0metricszstructures} (which in turn relies on the criterion from \cite{FJC_MCG}). In this paper we focus on applications to the study of ``Dehn filling like" quotients, of the type considered in \cite{dfdt, BHMS,short_HHG:II}. In particular, we prove the following, where $S_n$ denotes the $n$-punctured sphere and $B_n$ is the braid group on $n$ strands.

\begin{thmintro}\label{thmintro:distinguish}
    \begin{enumerate}
    \item Let $n\geq 5$ be an integer. Then there exists a non-elementary hyperbolic quotient $Q$ of $B_4$ such that any homomorphism $\phi\colon B_n\to Q$ has virtually cyclic image.
    \item Let $n\geq 6$ be an integer. Then there exists a non-elementary hyperbolic quotient $Q$ of $\MCG(S_5)$ such that any homomorphism $\phi\colon\MCG(S_n)\to Q$ has finite image.
    \end{enumerate}
\end{thmintro}

\noindent \Cref{thmintro:distinguish} gives a strong way to distinguish $B_4$ (resp. $\MCG(S_5)$) among all other braid groups (resp. mapping class groups of spheres) using infinite hyperbolic quotients; this is in analogy with profinite rigidity results, where finite quotients are used instead. We emphasise that \Cref{thmintro:distinguish} is extremely delicate: for instance, a \emph{finite-index subgroup} of $B_n$ for $n\geq 5$ surjects onto $B_4$ (by forgetting strands), whence onto any quotient of it. This subtlety is reflected in the proof of Theorem \ref{thm:ubertheorem}, which implies both statements of \Cref{thmintro:distinguish} and crucially uses properties of finite subgroups of $\MCG(S_n)$.

The quotients $Q$ arising in \Cref{thmintro:distinguish} are in fact all virtually of the form $\MCG^\pm(S_5)/DT_N$, where $DT_N$ is the normal subgroup generated by all $N$-th powers of Dehn twists. For suitable $N$, these were shown to be hyperbolic in \cite{dfdt}. Understanding torsion in these groups is the key to proving the theorem above, and we can also exploit this knowledge to control their automorphism groups.

\begin{thmintro}\label{thmintro:automorphisms}
    There exists $N_0>0$ such that, for all multiples $N$ of $N_0$, the group $G_N=\MCG^\pm(S_5)/DT_N$ satisfies $\Aut{G_N}\cong G_N$.
\end{thmintro}

In \cite{rigidity_mcg_mod_dt}, we proved an analogue of \Cref{thmintro:automorphisms} for $\MCG^\pm(S_n)/DT_N$ where $n\ge 7$, but the proof there is completely different in nature (it first establishes quasi-isometric rigidity, relying on a sufficiently rich pattern of quasiflats and results from \cite{quasiflats}, and then uses that automorphisms give quasi-isometries). Though the remaining case $n=6$ cannot be tackled with the tools from \cite{rigidity_mcg_mod_dt}, one can likely use methods similar to the ones we use here to settle it, provided that a certain technical improvement on \cite{BHMS} can be achieved. We believe it can, and we explain this further in \Cref{rmk:Giorgios_homework}.

We also point out that, by combining existing results in the literature and using completely different tools, one could already deduce that the outer automorphism group of $G_N$ is \emph{finite} (see \Cref{rem:out_is_finite}). Roughly, property (FA) for $\MCG^\pm(S_5)$ implies that the $G_N$ are rigid in the sense of JSJ theory, but this type of argument cannot prove that outer automorphism groups are trivial rather than finite.

In the case of relatively hyperbolic groups, sequences of peripheral quotients can be studied via the corresponding quotients of cusped spaces, which turn out to be uniformly hyperbolic; this is exploited for instance in \cite{dahmaniguirardel}. Inspired by this, we consider sequences of Dehn filling quotients of \emph{short HHGs}, a particularly simple class of HHG introduced in \cite{short_HHG:I}. Notably, $\MCG^\pm(S_5)$ is a short HHG, and the quotients $G_N$ from \Cref{thmintro:automorphisms} are examples of Dehn filling quotients (see \Cref{notation:short_G_and_DT}). The following result, which is the starting point towards all the aforementioned applications, establishes that the corresponding quotients of cusped spaces are ``uniform", in analogy with the relatively hyperbolic case. We refer to \Cref{rem:uniform_CHHS} for details of the notion of ``uniformly HHS''. 

\begin{thmintro}\label{thmintro:uniform}
    Let $G$ be a short HHG, with cusped space $\mathcal Z$. There exists $N_0>0$, such that the spaces $\mathcal Z/DT_N$ for all multiples $N$ of $N_0$ are uniformly HHS.
\end{thmintro}

\noindent It was proven in \cite{HHP:coarse}, using injective spaces, that all HHGs have bounded torsion. The proofs in fact give a bound depending on the underlying hierarchical data, so if a family of groups acts uniformly properly on a uniform family of HHS (such as the collection $\mathcal Z/DT_N$ from \Cref{thmintro:uniform}) then the bound on torsion will be uniform. Because of this, we are able to show the following, which we state for $\MCG(S_5)$ only for simplicity even though an analogue holds for all short HHGs.

\begin{corintro}
\label{corintro:torsion-in-quotients}
    There exists $N_0>0$ such that, for all multiples $N$ of $N_0$, any finite-order element of $\MCG(S_5)/DT_N$ is the image of an element of $\MCG(S_5)$ which is either finite-order, stabilises a curve, or swaps two disjoint curves.
\end{corintro}

\noindent We expect that a similar result can be obtained for all quotients of mapping class groups by suitable powers of Dehn twists, as well as the further quotients constructed in \cite{BHMS}. We explain the currently missing technical statement in \Cref{rmk:Giorgios_homework} below.

\subsection*{Outline and strategies}
In \Cref{sec:background} we recall necessary background, in particular about the combinatorial approach to hierarchical hyperbolicity from \cite{BHMS}.  Very roughly, one can extract a hierarchical structure from hyperbolic simplicial complexes where all links are also hyperbolic. Inspired by \cite{converse}, we consider simplicial complexes obtained via the following ``blowup" procedure: starting from a ``supporting" simplicial complex $\ov X$, every vertex $v\in \ov X^{(0)}$ is replaced by the cone over a certain graph $L_v$, thus ensuring that $L_v$ appears as a link in the blown-up graph. This is a fairly broad setting, as a lot of naturally occurring HHSs admit a combinatorial structure of this type by results from \cite{converse}. The associated hierarchically hyperbolic space (HHS) is then a graph whose vertices are maximal simplices of the blown-up graph.  For mapping class groups, the supporting complex is the curve complex and the $L_v$'s are annular curve graphs, so that maximal simplices of the blown-up complex correspond to markings.

In \Cref{sec:squid} we describe the (relative) combinatorial HHSs that we will consider, which we call (relative) \squid{} HHS and as mentioned above are blow-ups. See Figure \ref{fig:HHS_plus_adjectives} to visualise where this notion fits among other version of hierarchical hyperbolicity. \Cref{sec:squid} also contains the construction of cusped spaces, which is done by replacing each $L_v$ with the combinatorial horoball over it. \Cref{thm:teich_is_hhs} summarises the main properties of this construction, and in particular proves the first two items of \Cref{thmintro:first}. It is possible to extend the construction beyond (relative) \squid{} HHGs, but this requires different techniques and will be done in \cite{HPZ_in_progress}.

In \Cref{sec:examples} we verify that our construction indeed applies to many (relative) HHSs, and that it yields cusped spaces of relatively hyperbolic groups and Teichm\"{u}ller space (the last two items of \Cref{thmintro:first}). There we also prove \Cref{prop:different_proj_in_converse_are_equal}, which is of independent interest and could prove useful elsewhere as well. Roughly, the Proposition says that, when starting from a HHS structure and passing to a combinatorial version of it, as in \cite{converse}, the associated projections between domains remain coarsely the same.

In \Cref{sec:control_torsion} we establish \Cref{thm:torsion-of_quot}, which allows us to control torsion elements in quotients of (relative) HHGs with controlled geometry. As mentioned above, this uses arguments from \cite{HHP:coarse}, applied to the action on our cusped spaces.

We leverage our control on torsion in \Cref{sec:Aut}, where we study automorphisms of quotients of short HHGs. Towards this, we firstly show \Cref{thmintro:uniform}(=\Cref{cor:uniform_teich_for_quot}); the latter, combined with \Cref{thm:torsion-of_quot}, immediately yields \Cref{cor:hyp_satisfied}, which is a more general version of \Cref{corintro:torsion-in-quotients} about finite subgroups in quotients. 

Having established this, the goal of \Cref{thm:extract} is to extract a simplicial automorphism of a certain graph from a group automorphism; for $\MCG(S_5)/DT_N$, this graph is the quotient of the curve graph by $DT_N$. To this extent, the strategy (for $\MCG(S_5)$) is roughly to characterise images of Dehn twists in $\MCG(S_5)/DT_N$ as those elements that have a sufficiently large finite order, and a sufficiently large centraliser (we need the latter to rule out images of multi-twists). 

In turn, every automorphism of the quotient of the curve graph is induced by a mapping class, by an analogue of Ivanov's theorem we proved in \cite{rigidity_mcg_mod_dt}, so \Cref{thmintro:automorphisms}(=\Cref{cor:autmcgs5}) follows from \Cref{thm:extract}.

In \Cref{sec:spheres} we prove \Cref{thm:ubertheorem} and both parts of \Cref{thmintro:distinguish} as corollaries, again exploiting knowledge about torsion, and in particular using that torsion elements have ``very large'' normal closure in mapping class groups of spheres. This should be compared to several results in the literature where torsion is used to constrain homomorphisms between mapping class groups \cite{HarveyKorkmaz,ChenLanier}.


\begin{remintro}[\Squid{} structures for Dehn twist quotients]
    \label{rmk:Giorgios_homework}
Let $S$ be any finite-type surface, and let $\MCG(S)$ be its mapping class group. It is shown in \cite{BHMS} that $\MCG(S)/DT_N$, for $N$ deep enough, is a combinatorial HHG. However, the structure found there is not a \squid{} HHG structure, meaning a structure as in \Cref{sec:squid}: this is because the underlying graph $X$, which is the corresponding quotient of the curve complex, is not a blowup. Nonetheless, one can blow up each vertex of $X$ to a cone over a quotient of an annular curve graph. We expect this should yield a relative \squid{} HHG structure with uniform constants over all $N$ large enough; this would almost immediately give generalisations of \Cref{thmintro:uniform} (on cusped spaces being uniformly HHS) and \Cref{corintro:torsion-in-quotients} (about torsion elements in quotients). Most likely, this would also allow one to extend \Cref{thmintro:automorphisms} (on automorphisms of quotients) to the six-punctured sphere; recall that the cases with more punctures are covered by \cite{rigidity_mcg_mod_dt}. Finally, it would also be a starting point for more general versions of \Cref{thmintro:distinguish} (on different mapping class groups having different hyperbolic quotients); however, since quotients by large powers of Dehn twists are in general not hyperbolic, a one-to-one analogue of \Cref{thmintro:distinguish} would probably require uniform relative \squid{} HHG structures also for the further quotients studied in \cite{BHMS}.
\end{remintro}




\subsection*{Acknowledgements}
We thank Arthur Bartels and R\'{e}mi Coulon for suggesting possible future applications of cusped spaces, Alex Wright for several insightful questions, and Matt Durham for many useful suggestions. We are also grateful to Mark Hagen, Harry Petyt, and Abdul Zalloum for discussions on the generalisation of cusped spaces beyond combinatorial HHG.

\section{Background}
\label{sec:background}
In this Section, we gather all definitions and tools from the world of (combinatorial hierarchical) hyperbolicity that we shall need. We first recall here a construction of Groves-Manning \cite{GrovesManning}, see also \cite{Bow-relhyp}, inspired by horoballs in real hyperbolic spaces.

\begin{defn}\label{def:horoball}
Let $\Gamma$ be a simplicial graph. The \emph{combinatorial horoball} $\Hor{\Gamma}$ with base $\Gamma$ is the simplicial graph whose vertex set is $\Gamma^{(0)}\times \N$, and with the following two types of edges:
\begin{itemize}
\item For every $p\in \Gamma^{(0)}$ and $n\in \N$, $(p,n)$ is adjacent to $(p, n+1)$;
\item For every $p,q\in \Gamma^{(0)}$ and $n\in \N$, if $\dist_{\Gamma}(p,q)\le 2^n$ then $(p,n)$ is adjacent to $(q,n)$.
\end{itemize}
We often denote $(p,n)\in \Gamma^{(0)}\times \N$ simply by $p^n$. 
\end{defn}

\noindent The base $\Gamma$ is naturally contained in the corresponding combinatorial horoball, but it is not quasi-isometrically embedded into it. Rather, it is coarsely embedded, in the sense of the following definition:

\begin{defn}[Coarse embedding]\label{def:coarse_emb}
    Given two metric spaces $(X,\dist_X)$ and $(Y,\dist_Y)$ and positive real functions $m,M\coloneq \mathbb{R}_{\ge 0}\to \mathbb{R}_{\ge 0}$ such that $m(t)\xrightarrow[t\to +\infty]{}+\infty$, a map $f\colon X\to Y$ is a \emph{coarse embedding} if for every $x,x'\in X$, $$m\left(\dist_X(x,x')\right)\le \dist_Y(f(x),f(x'))\le M\left(\dist_X(x,x')\right).$$
    $M$ and $m$ are called the \emph{coarse embedding functions} of $f$.
\end{defn}

\noindent In real hyperbolic spaces every horosphere is a coarsely embedded copy of the Euclidean space inside the corresponding horoball. Analogously, combinatorial horoballs satisfy the following:

\begin{lem}[{See proof of \cite[Lemma 3.10]{GrovesManning}}]\label{lem:coarseemb_in_hor}
    The inclusion $\Gamma\hookrightarrow \Hor{\Gamma}$ is a uniform coarse embedding. More precisely, for every $x,y\in \Gamma^{(0)}$, $\dist_{\Hor{\Gamma}}(x,y)\ge \frac{2}{3}\log_{2}(\dist_{\Gamma}(x,y))+1$.
\end{lem}

\noindent  Geodesics in combinatorial horoballs admit a similar description as those in real horoballs, and one can use this to show:

\begin{thm}[{\cite{GrovesManning}}]\label{thm:horgamma_hyp}
There exists $\delta_0>0$ such that, for every $\Gamma$, $\Hor{\Gamma}$ is $\delta_0$-hyperbolic.
\end{thm}

\noindent To summarise, combinatorial horoballs give a way to coarsely embed any graph into a hyperbolic space.

\subsection{Pills of combinatorial hierarchical hyperbolicity} We recall here the notion of a \emph{combinatorial HHS} from \cite{BHMS}. For the whole section, let $X$ be a simplicial graph, with vertex set $\ov X^{(0)}$, and let $\W$ be an \emph{$X$-graph}, that is, a simplicial graph whose vertex set is the set of all maximal simplices of $X$. We conflate $X$ with the flag simplicial complex that has $X$ as the 1-skeleton, and in particular often refer to cliques in $X$ as simplices of $X$. The main theorem of \cite{BHMS} describes conditions on a pair $(X,\W)$ to ensure that $\W$ is a hierarchically hyperbolic space. We now gather the required terminology to phrase these conditions.

Given two subgraphs $A,B$ of $X$ (by which we always mean \emph{full} subgraphs), if $B\subseteq \link{X}{A}$ we write $A\star B$ to denote the full subgraph spanned by $A^{(0)}\cup B^{(0)}$. Moreover, we set $\CStar{A}\coloneq \link{X}{A}\star A$.

Links of simplices will correspond to hyperbolic spaces of an HHS structure. Because of this, we are interested in simplices that have the same link.

\begin{defn}[Saturation]\label{defn:simplex_equivalence}
For $\Delta,\Delta'$ simplices of $X$, we write $\Delta\sim\Delta'$ to mean $\link{X}{\Delta}=\link{X}{\Delta'}$. Let $[\Delta]$ be the $\sim$--equivalence class of $\Delta$, and let 
$$\Sat(\Delta)=\left(\bigcup_{\Delta'\in[\Delta]}\Delta'\right)^{(0)}.$$
Let $\frakS$ the set of $\sim$--classes of non-maximal simplices in $X$. 
\end{defn}

The hyperbolic spaces of the HHS structures we will be interested in are not exactly links in $X$ (in most cases there are discrete links), but rather we have to add certain edges, as specified below.

\begin{defn}[Augmented links]\label{defn:complement}
Let $\duaug{X}{\mathcal{W}}$ be the simplicial graph such that:
\begin{itemize}
     \item the $0$--skeleton of $\duaug{X}{\mathcal{W}}$ is $X^{(0)}$;
     \item if $v,w\in X^{(0)}$ are adjacent in $X$, then they are adjacent in $\duaug{X}{\mathcal W}$; 	
     \item if two vertices in $\mathcal W$ are adjacent, then we consider $\sigma,\rho$, the associated maximal simplices of $X$, and in $\duaug{X}{\mathcal W}$ we connect each vertex of $\sigma$ to each vertex of $\rho$.
\end{itemize}
 For each simplex $\Delta$ of $X$, let $Y_\Delta$ be the subgraph of $\duaug{X}{\mathcal W}$ spanned by $(\duaug{X}{\mathcal W})^{(0)}-\Sat(\Delta)$. The \emph{augmented link} of $\Delta$ is the induced subgraph $\C{\Delta}$ of $Y_\Delta$ spanned by $\link{X}{\Delta}^{(0)}$ (we emphasise that we are taking links in $X$, not in $\duaug{X}{\mathcal W}$, and then considering the subgraphs of $Y_\Delta$ induced by those links). Note that, if $\Delta\sim\Delta'$, then $Y_{\Delta}=Y_{\Delta'}$ and  $\C{\Delta}=\C{\Delta'}$ . 
\end{defn}

We are ready to define combinatorial HHS.

\begin{defn}[Combinatorial HHS]\label{def:combHHS}
A \emph{combinatorial hierarchically hyperbolic space} is the data of a simplicial graph $X$, an $X$-graph $\W$, and two constants $n,\delta\ge 0$, satisfying the following:
\begin{enumerate}[label=(\arabic*)]
\item \label{CHHS_def_complexity} Any chain $\link{X}{\Delta_1}\subsetneq\dots\subsetneq\link{X}{\Delta_k}$ has length at most $n$, which is called the \emph{complexity} of $(X,\W)$.
\item \label{CHHS_def_hyperbolicity} For each $[\Sigma]\in \frakS$,  $\C{\Sigma}$ is $\delta$--hyperbolic.
\item \label{CHHS_def_qi_emb} For each $[\Sigma]\in \frakS$, $\C{\Sigma}$ is $\delta$--quasi-isometrically embedded in $Y_\Sigma$.
\item \label{CHHS_def_lk_cap_lk} For every $[\Delta],[\Sigma]\in \frakS$ for which there exists $[\Gamma]\in\frakS$ such that $\link{X}{\Gamma}\subseteq \link{X}{\Delta}\cap \link{X}{\Sigma}$ and $\diam(\C{\Gamma})\geq \delta$, there exists a non-maximal simplex $\Pi$ containing $\Sigma$ such that $\link{X}{\Pi}\subseteq \link{X}{\Delta}$, and all $\Gamma$ as above satisfy $\link{X}{\Gamma}\subseteq \link{X}{\Pi}$.
\item \label{CHHS_def_edge_in_link} For each $[\Delta]\in \frakS$ and $v\neq w\in \link{X}{\Delta}$, if $v,w$ are adjacent in $\C{\Delta}$ but not in $\link{X}{\Delta}$, then there exist $\mathcal W$-adjacent maximal simplices $\Pi_v,\Pi_w$ such that $\Delta\star \{v\}\subseteq \Pi_v$ and $\Delta\star \{w\}\subseteq \Pi_w$.
\end{enumerate}
\end{defn}

\noindent \cite[Theorem 1.18]{BHMS} states that, if $(X,\W,n,\delta)$ is a combinatorial HHS, then $\W$ is a \emph{hierarchically hyperbolic space}, in the sense of \cite[Definition 1.1]{HHS_II}. More precisely, the structure on $\W$ has $\frakS$ as domain set and augmented links as coordinate spaces, while relations and projections between domains are encoded by the following additional definitions.

\begin{defn}[Relations between domains]\label{defn:nest_orth} Let $[\Delta],[\Delta']\in\frakS$. Then:
\begin{itemize}
     \item $[\Delta]\nest[\Delta']$ if $\link{X}{\Delta}\subseteq\link{X}{\Delta'}$;
     \item $[\Delta]\orth[\Delta']$ if $\link{X}{\Delta'}\subseteq \link{X}{\link{X}{\Delta}}$.
\end{itemize}
If $[\Delta]$ and $[\Delta']$ are neither $\orth$--related nor $\nest$--related, we write 
$[\Delta]\transverse[\Delta']$.
\end{defn}

\begin{defn}[Projections]\label{defn:projections}
Fix $[\Delta]\in\frakS$ and define a map $\pi_{[\Delta]}\colon \W\to 2^{\C{\Delta}}$ as follows. Let $p\colon Y_\Delta\to2^{\C{\Delta}}$ be the coarse closest point projection, i.e. 
$$p(x)=\{y\in\C{\Delta}\colon \dist_{Y_\Delta}(x,y)\le\dist_{Y_\Delta}(x,\C{\Delta})+1\}.$$
If $w$ is a maximal simplex of $X$, the intersection $w\cap Y_\Delta$ is non-empty and has diameter at most $1$ by \cite[Lemma 1.15]{BHMS}. Hence set $$\pi_{[\Delta]}(w)=p(w\cap Y_\Delta),$$
and for every $w,w'\in \W$ set $$\dist_{[\Delta]}(w,w')\coloneq\dist_{\C{\Delta}}\left(\pi_{[\Delta]}(w),\pi_{[\Delta]}(w')\right).$$
If $[\Delta],[\Delta']\in\frakS$ satisfy $[\Delta]\transverse[\Delta']$ or $[\Delta']\propnest [\Delta]$, set $$\rho^{[\Delta']}_{[\Delta]}=p(\Sat(\Delta')\cap Y_\Delta).$$ 
If $[\Delta]\propnest [\Delta']$, let $\rho^{[\Delta']}_{[\Delta]}\colon \C{\Delta'}\to \C{\Delta}$ be defined as follows.  On $\C{\Delta'}\cap Y_\Delta$, it is the restriction of $p$; otherwise, it takes the value $\emptyset$.
\end{defn}

\begin{rem}[Uniform combinatorial HHSs]\label{rem:uniform_CHHS}
    It can be extracted from the proof of \cite[Theorem 1.18]{BHMS} that, for a combinatorial HHS $(X,W,n,\delta)$, all constants and functions involved in the hierarchically hyperbolic structure depend only on $n$ and $\delta$ (in particular, the function appearing in the Uniqueness axiom is explicitly constructed at \cite[page 50]{BHMS}). This is relevant as every coarse object naturally associated to a hierarchically hyperbolic space, such as its coarse median structure \cite{HHS_II} or its coarsely injective structure \cite{HHP:coarse}, ultimately only depends on these constants and functions. With this in mind, we say that a family of combinatorial HHSs is \emph{uniform} if they satisfy \Cref{def:combHHS} with respect to uniform constants $n$ and $\delta$.
\end{rem}

\subsection{Blowup graphs}

All the combinatorial HHSs that we consider in this paper will be blowups, in a sense that we now describe. We note that the combinatorial HHSs constructed in \cite{converse}, where a combinatorial structure is given for many HHSs, are also all blowups: this will be explored further in \Cref{sec:converse_are_squid} below.

\begin{defn}\label{defn:blowup_data}
Given a simplicial graph $\ov X$ and, for every $v\in \ov X^{(0)}$, a set $L_v$, the \emph{blowup} of $\ov X$ with respect to $\{L_v\}_{v\in \ov X^{(0)}}$ is the graph $X$ defined as follows. For every $v\in \ov X^{(0)}$, $X$ contains the \emph{cone} $\Cone{v}\coloneq \{v\}\star\{L_v\}$. Given $v,w\in \ov X^{(0)}$, $\Cone{v}$ and $\Cone{w}$ span a join in $X$ if their \emph{apices} $v$ and $w$ are adjacent in $\ov X$, and are disjoint otherwise. See \Cref{fig:blowup}. 

We call $\ov X$ the \emph{support graph}, $L_v$ the \emph{base} of the cone below $v\in \ov X^{(0)}$, and the pair $(\ov X, \{L_v\}_{v\in \ov X^{(0)}})$ the \emph{blowup data}. When the blowup data is not relevant or clear from context, we say that $X$ is a \emph{blowup graph}.
\end{defn}

\begin{figure}[htp]
    \centering
    \includegraphics[width=\textwidth, alt={The cones of two adjacent vertices span a join}]{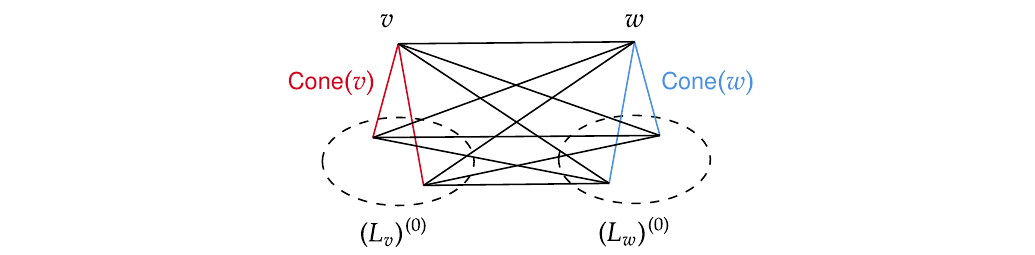}
    \caption{The blowup of two adjacent vertices of $\ov{X}$.}
    \label{fig:blowup}
\end{figure}

\begin{notation}
If $X$ is the blowup of $(\ov X, \{L_v\}_{v\in \ov X^{(0)}})$, we identify $\ov X$ with the subgraph of $X$ spanned by the apices of the cones. Let $\p\colon X\to \ov X$ be the $1$-Lipschitz retraction mapping every cone to its apex. If $\Sigma$ is a simplex of $X$, we denote $\p(\Sigma)$ by $\ov \Sigma$, which is the \emph{support} of $\Sigma$.

Given a simplex $\ov \Theta$ of $\ov X$, together with a point $p\in L_v$ for every $v\in \ov \Theta^{(0)}$, we denote the $\hat X$-simplex spanned by $\{v,p\}_{v\in \ov\Theta}$ by $\Theta\coloneq (v,p)_{v\in \ov\Theta}$. If $\ov\Theta$ is maximal in $\ov X$ then $\Theta$ is maximal in $X$, and all maximal simplices of $X$ are of this form.

For every $v\in \ov X^{(0)}$ we denote by $\Delta_v$ any simplex of $X$ whose link is $L_v$. Notice that $L_v$ is discrete, so $[\Delta_v]$ is $\nest$-minimal.
\end{notation}

Links in a blowup admit a natural decomposition in terms of blowups of links and parts of cones, as described in the following lemma.

\begin{lem}[{Decomposition of links \cite[Lemma 4.3]{converse}}]\label{lem:decomposition_of_links}
Let $\Sigma$ be a simplex of $X$. Then
$$\link{X}{\Sigma}=\p^{-1}\left(\link{\ov X}{\ov \Sigma}\right)\star \left(\bigstar_{v\in\ov\Sigma^{(0)}}\link{\Cone{v}}{\Sigma\cap\Cone{v}}\right).$$
\end{lem}
\begin{cor}\label{cor:bounded_links}
Any $X$-simplex $\Sigma$ is of one of the following types:
\begin{itemize}
 \item \label{cor:bounded_links1} \emph{Bounded type}: $\link{X}{\Sigma}$ is either a single vertex or a non-trivial join.
 \item \label{cor:bounded_links2} \emph{Blowup type}: for each $v\in \ov\Sigma^{(0)}$, we have that $\Sigma\cap \Cone{v}$ is an edge.
 \item \label{cor:bounded_links3} \emph{Cone type}: $\Sigma=\Delta_v$ for some $v\in \ov X^{(0)}$.
\end{itemize}
\end{cor}

\noindent The following feature of intersections of links is a common property for combinatorial HHSs (for example, it holds for curve graphs of surfaces, as explained in \cite[Remark 6.6]{BHMS}).

\newcommand{\cleanish}[1]{cleanish intersections}
\begin{defn}
    A simplicial graph $\Gamma$ has \emph{\cleanish{}} if, for every two simplices $\Sigma,\Phi$ of $\Gamma$, there exist two simplices $\Pi,\Psi$ such that $\Sigma\subseteq \Pi$ and 
    $$\link{\Gamma}{\Sigma}\cap \link{\Gamma}{\Phi}=\link{\Gamma}{\Pi}\star \Psi.$$
\end{defn}

\noindent We conclude this section with a technical lemma saying that blowups preserve \cleanish{}.
\begin{lem}\label{lem:cleanish_for_blowup}
    Let $X$ be a blowup of a graph $\ov X$. If $\ov X$ has \cleanish{} then $X$ has \cleanish{}.
\end{lem}

\begin{proof}
    Let $\Sigma, \Phi$ be simplices of $X$. We want to find two simplices $\Pi,\Psi$ such that $\Sigma\subseteq \Pi$ and
    $\link{X}{\Sigma}\cap \link{X}{\Phi}=\link{X}{\Pi}\star \Psi$. 
    Since $$\link{X}{\Sigma}\cap \link{X}{\Phi}=\link{X}{\Sigma\star(\Phi\cap\link{X}{\Sigma})}\cap \link{X}{\Phi},$$ 
    we first replace $\Sigma$ by $\Sigma'\coloneq \Sigma\star(\Phi\cap\link{X}{\Sigma})$. Now let $\ov\Psi$ be a (possibly trivial) $\ov X$-simplex inside $\link{\ov X}{\ov\Sigma'}\cap \link{\ov X}{\ov\Phi}$ such that $$\link{\ov X}{\ov \Sigma'}\cap \link{\ov X}{\ov \Phi}\subseteq \CStar{\ov \Psi},$$
    and which is maximal with this property. Then
    $$\link{\ov X}{\ov \Sigma'}\cap \link{\ov X}{\ov \Phi}=\left( \link{\ov X}{\ov \Sigma'\star\ov\Psi}\cap \link{\ov X}{\ov \Phi}\right)\star\ov\Psi.$$
    Since $\ov X$ has \cleanish{}, there exist two simplices $\ov\Theta,\ov\Psi'\subseteq \link{\ov X}{\ov\Sigma'\star\ov \Psi}$ such that 
    $$\link{\ov X}{\ov \Sigma'\star\ov\Psi}\cap \link{\ov X}{\ov \Phi}=\link{\ov X}{\ov \Sigma'\star\ov\Psi\star\ov \Theta}\star\ov \Psi',$$
    and by maximality of $\ov\Psi$ we must have that $\Psi'=\emptyset$. Summing up, we proved that
    \begin{equation}\label{eq:lk_cap_lk_in_ovX}
        \link{\ov X}{\ov \Sigma'}\cap \link{\ov X}{\ov \Phi}=\link{\ov X}{\ov \Sigma'\star\ov\Psi\star\ov\Theta}\star\ov\Psi.
    \end{equation}
    We now construct the simplices $\Pi$ and $\Psi$ that we need. Let $\Psi=\ov\Psi$, seen as a simplex of $X$. Let $\ov\Pi=\ov \Sigma'\star\ov\Psi\star\ov\Theta$, and define a simplex $\Pi$ with support $\ov\Pi$ as follows:
    \begin{itemize}
        \item If $v\in \ov\Sigma'-\CStar{\ov\Phi}$, let $\Pi\cap \Cone{v}$ be an edge containing $\Sigma'\cap \Cone{v}$, so that $\link{\Cone{v}}{\Pi\cap \Cone{v}}=\emptyset$;
        \item If $v\in \ov\Sigma'\cap\CStar{\ov\Phi}$, let $\Pi\cap \Cone{v}=\Sigma'\cap \Cone{v}$. In particular, $\link{\Cone{v}}{\Pi\cap \Cone{v}}$ and $\Cone{v}\cap\link{X}{\Sigma'}\cap\link{X}{\Phi}$ coincide;
        \item If $v\in \ov\Theta$, let $\Pi\cap \Cone{v}$ be an edge;
        \item If $v\in \ov\Psi$, let $\Pi\cap \Cone{v}=\{v\}$, so that $\link{\Cone{v}}{\Pi\cap \Cone{v}}=L_v$.
    \end{itemize}

\begin{figure}[htp]
\centering
\renewcommand{\arraystretch}{1.5}
\begin{tabular}{c|c|c}
$\Pi\cap \Cone v$&$\ov\Sigma'$&$\link{\ov X}{\ov \Sigma'}$\\
\hline
$\ov\Phi$&  $\Sigma'\cap \Cone v$&/\\
$\link{\ov X}{\ov \Phi}$&$\Sigma'\cap \Cone v$& $v$ whenever $v\in \ov\Psi$\\
$\ov X- \CStar{\ov\Phi}$&complete $\Sigma'\cap \Cone v$ to an edge& choose an edge whenever $v\in \ov\Theta$
\end{tabular}
    \caption{Schematic representation of $\Pi$. Each cell describes how $\Pi\cap \Cone v$ is defined whenever $v\in\ov X^{(0)}$ belongs to the area given by the intersection between the row label and the column label: for example, if $v\in\ov\Sigma'\cap \ov\Phi$ we have that $\Pi\cap \Cone v=\Sigma'\cap \Cone v$.}
    \label{tab:Sigma_star_Pi}
\end{figure}

\noindent It is clear by the first two bullets that $\Pi$ extends $\Sigma'$, so we are left to show that $$\link{X}{\Phi}\cap\link{X}{\Sigma'}=\link{X}{\Pi}\star\Psi.$$ Firstly $\Psi\subseteq\link{X}{\Phi}\cap\link{X}{\Sigma'}$ by construction. Next, let $u\in \link{X}{\Pi}\subseteq \link{X}{\Sigma'}$, and we claim that $u\in \link{X}{\Phi}$. By inspection of \Cref{lem:decomposition_of_links}, we must have that either $\p(u)\in \ov\Sigma'$, or $\p(u)\in \link{\ov X}{\ov\Sigma'\star\ov \Theta}$. In the first case, a careful inspection of how we defined $\Pi$ shows that $u\in \link{ X}{\Phi}$, and moreover $u\in \link{X}{\Psi}$ since $\ov\Sigma'\in \link{\ov X}{\ov \Psi}$. Otherwise $\p(u)\in \link{\ov X}{\ov\Sigma'\star\ov \Theta}\subset\link{\ov X}{\ov\Phi}$, and therefore $u\in\link{X}{\Phi}$; this also shows that $u\in \link{\ov X}{\ov\Sigma'}\cap\link{\ov X}{\ov\Phi}\subseteq\CStar{\ov\Psi}$, so that again $u\in \link{X}{\Psi}$. We just showed that $\link{X}{\Pi}$ and $\Psi$ span a join, and moreover that $\link{X}{\Phi}\cap\link{X}{\Sigma'}\supseteq \link{X}{\Pi}\star\Psi$.

For the converse inclusion, let $u\in \link{X}{\Phi}\cap\link{X}{\Sigma'}$, so that $\p(u)$ belongs to $\CStar { \ov \Sigma'}\cap \CStar { \ov \Phi}$. There are two possible cases to consider:
\begin{itemize}
    \item If $\p(u)\in \ov \Sigma'$ then $\link{\Cone{v}}{\Pi\cap \Cone{v}}=\Cone{v}\cap\link{X}{\Sigma'}\cap\link{X}{\Phi}$ by how we constructed $\Pi$, so $u\in \link{\Pi}{\Psi}$.
    \item Otherwise $\p(u)\in  \link{\ov X}{\ov \Sigma'}\cap \link{\ov X}{\ov \Phi}=\link{\ov X}{\ov \Pi}\star\ov\Psi$. If $\p(u)\in \link{\ov X}{\ov \Pi}$ then $u\in \link{X}{\Pi}$. Otherwise $\p(u)\in\ov\Psi$, so either $u\in \ov\Psi=\Psi$, or $u\in L_{\p(u)}$, which coincides with $\link{\Cone{\p(u)}}{\Pi\cap \Cone{\p(u)}}$ by construction. In both cases $u\in \link{X}{\Pi}\star \Psi$, as required. \qedhere
\end{itemize}
\end{proof}

\section{Cusped spaces for \squid{} HHSs}
\label{sec:squid}
We will be able to construct a cusped space for any relative HHS with a nice combinatorial structure. We make this precise in the next definition.

\subsection{\Squid{} HHS}
\begin{defn}\label{defn:rel_squid_HHS}
A \emph{\squid{} HHS} is a combinatorial HHS $(X,\W,n,\delta)$, where $X$ is a blowup of a graph $\ov X$ with \cleanish{}. A \emph{relative \squid{} HHS} is defined analogously, except that \Cref{def:combHHS}.\ref{CHHS_def_hyperbolicity} is replaced by:
\begin{enumerate}[start=2, label = (\arabic*{}')]
\item \label{CHHS_def_rel_hyperbolicity} For each non-maximal simplex $\Sigma$, either $\C{\Sigma}$ is $\delta$-hyperbolic, or $\Sigma$ is of the form $\Delta_v$ for some $v\in \ov X^{(0)}$.
\end{enumerate}
We often drop the constants $n$ and $\delta$ when they are not relevant, and refer to a (relative) \squid{} HHS simply by $(X,\W)$. 
\end{defn}
\begin{rem}\label{rem:dimension_of_ovX}
    Notice that, for every simplex $\ov\Sigma$ of $\ov X$, there exists a simplex $\Theta$ of $X$ such that $\link{X}{\Theta}=\ov\Sigma$. This shows that the dimension of $\ov X$ is bounded above by the constant $n$, since every chain of inclusions of simplices in $\ov X$ induces a chain of inclusions of links of simplices of $X$.
\end{rem}

\noindent There is a natural notion of action on a \squid{} HHS, leading to the definition of a \squid{} HHG:

\begin{defn}\label{defn:rel_squid_HHG}
Let $(X,\W)$ be a (relative) \squid{} HHS, with support graph $\ov X$, and let $G$ be a finitely generated group. A $G$-\emph{action} on $(X,\W)$ is a simplicial action of $G$ on $X$ such that:
\begin{itemize}
    \item The action preserves $\ov X$;
    \item There are finitely many $G$-orbits of simplices in $\ov X$;
    \item The action on $X$ induces a simplicial action on $\W$.
\end{itemize}
If moreover $G$ acts geometrically on $\W$, then we say $G$ is a \emph{(relative) \squid{} HHG}.
\end{defn}

The figure below summarises the relations among the various versions of hierarchical hyperbolicity relevant for this paper.

\begin{figure}[htp]
    \centering
    \includegraphics[width=\linewidth]{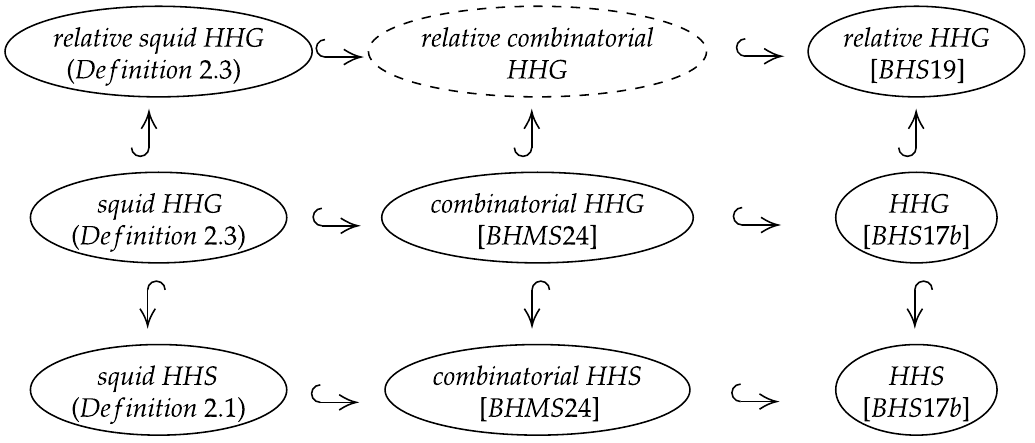}
    \caption{The various notions of hierarchical hyperbolicity at play in this paper, with arrows representing inclusions. Objects from the right column belong to the world of coarse geometry, while those in the middle and left columns are combinatorial in nature. At present, a robust definition of a relative combinatorial HHG has not been devised yet (hence we put this notion in a dashed box). Any reasonable notion should include both relative \squid{} HHGs and combinatorial HHGs, and provide examples of relative HHGs.}
    \label{fig:HHS_plus_adjectives}
\end{figure}

\begin{lem}\label{lem:finman_orbit_of_lk} Let $(X,\W)$ be a relative \squid{} HHS. If $G$ acts on $(X,\W)$ then there are finitely many $G$-orbits of links of simplices in $X$.
\end{lem}
\begin{proof}
     Recall that, by \Cref{lem:decomposition_of_links}, the link of a simplex $\Sigma$ is given by
    $$\link{X}{\Sigma}=\p^{-1}\left(\link{\ov X}{\ov \Sigma}\right)\star \left(\bigstar_{v\in\ov\Sigma^{(0)}}\link{\Cone{v}}{\Sigma\cap\Cone{v}}\right);$$
    hence $\link{X}{\Sigma}$ is uniquely determined by the support simplex $\ov\Sigma$, of which there are finitely many $G$-orbits, and by the choice of $\link{\Cone{v}}{\Sigma\cap\Cone{v}}$ for every ${v\in\ov\Sigma^{(0)}}$, which can either be:
    \begin{itemize}
        \item $L_v$ if $\Sigma\cap\Cone{v}=\{v\}$;
        \item $\{v\}$ if $\Sigma\cap\Cone{v}\in L_v$;
        \item empty if $\Sigma\cap\Cone{v}$ is an edge.
    \end{itemize}
    Therefore there are three possible choices for every ${v\in\ov\Sigma^{(0)}}$, and this concludes the proof that $G\circlearrowleft X$ has finitely many orbits of links.
\end{proof}

\noindent As a special case of {\cite[Theorem 1.18]{BHMS}}, we get:
\begin{cor}
    A \squid{} HHG is a hierarchically hyperbolic group.
\end{cor}

\subsection{Cusped space: the underlying graph}

We construct cusped spaces by modifying \squid{} HHSs in a natural way; the idea is simply to replace each minimal coordinate space $\C{\Delta_v}$ with the combinatorial horoball over it. We now describe this procedure, starting with the combinatorial aspects.

\begin{notation}\label{notation:pre_cusp}
Let $(X,\W)$ be a relative \squid{} HHS, with blowup data $(\ov X, \{L_v\}_{v\in \ov X^{(0)}})$. Define $\hat X$ as the blowup of $(\ov X, \{L_v\times \N\}_{v\in \ov X^{(0)}})$. We identify $X$ with the blowup of $(\ov X, \{L_v\times \{0\}\}_{v\in \ov X^{(0)}})$ inside $\hat X$. With a small abuse of notation, we still denote by $\p\colon \hat X\to \ov X$ the retraction mapping every cone to its apex. Let $\down{(\cdot)}\colon \hat X\to X$ map every $v\in \ov X^{(0)}$ to itself, and every $p^n\coloneq (p,n)\in L_v\times \N$ to $p\in L_v$. 

For every maximal simplex $\Theta=(v,{p}^{m_v})_{v\in \ov\Theta}$ of $\hat X$, let $\depth{\Theta}=\max_{v\in \ov\Theta} m_v$. We also set $\Theta^+=(v,{p}^{m_v})_{v\in \ov\Theta\mid m_v>0}$, which might be non-maximal.
\end{notation}

\begin{lem}\label{lem:links_in_cusped}
If $\Delta$ is a $\hat X$-simplex and $x\in \hat X^{(0)}$, then $x\in \link{\hat X}{\Delta}$ if and only if $\down{x}\in \link{X}{\down{\Delta}}$. In particular $\down{\link{\hat X}{\Delta}}= \link{X}{\down{\Delta}}$.
\end{lem}

\begin{proof}
By construction, two vertices $x,y\in \hat X^{(0)}$ are $\hat X$-adjacent if and only if either $\p(x)=\p(y)$ and one of $x$ and $y$ coincides with $\p(x)$, or $\p(x)$ and $\p(y)$ are $\ov X$-adjacent. The same characterisation holds for $X$; therefore, since $\p(\down{x})=\p(x)$, and similarly for $y$, we get that $x\in\link{\hat X}{y}$ if and only if $\down{x}\in\link{X}{\down{y}}$. \Cref{lem:links_in_cusped} now follows by taking the intersections over all $y\in  \Delta^{(0)}$.
\end{proof}

\begin{cor}\label{cor:nesting_of_links_in_cusp}
Let $\Delta,\Sigma$ be simplices of $\hat X$. Then $\link{\hat X}{\Delta}\subseteq \link{\hat X}{\Sigma} $ if and only if $\link{X}{\down{\Delta}}\subseteq \link{X}{\down{\Sigma}} $.
\end{cor}

\begin{proof} 
If $\link{\hat X}{\Delta}\subseteq \link{\hat X}{\Sigma} $ then $$\link{X}{\down{\Delta}}=\down{(\link{\hat X}{\Delta})}\subseteq \down{(\link{\hat X}{\Sigma})}=\link{X}{\down{\Sigma}},$$
where we invoked \Cref{lem:links_in_cusped} for the equalities. For the converse inclusion, assume that  $\link{X}{\down{\Delta}}\subseteq \link{X}{\down{\Sigma}} $, and let  $x$ be any point in $\link{\hat X}{\Delta}$. Then $\down{x}$ belongs to $\link{X}{\down{\Delta}}\subseteq \link{X}{\down{\Sigma}} $, and again $x\in\link{\hat X}{\Sigma}$ by \Cref{lem:links_in_cusped}.
\end{proof}

\subsection{Cusped space: edges between maximal simplices}

Our goal is now to describe the graph $\hat \W$ that makes $(\hat X,\hat\W)$ into a combinatorial HHS. Roughly, $\hat \W$ ``contains'' the combinatorial horoballs over the $\nest$-minimal domains of $\W$.

\begin{defn}[Cusped space]\label{defn:cusped_space}
In the setting of \Cref{notation:pre_cusp}, let $\hat \W$ be the $\hat X$-graph where two maximal simplices $\Theta,\Xi\subseteq\hat X$ span a $\hat \W$-edge if and only if one of the following happens:
 
\begin{enumerate}[label=(\alph*)]
\item\label{hatW-edge_cusp} \emph{Edge of cusp-type}: There exist $v\in \ov \Theta^{(0)}$ and $p^n,q^m\in L_v\times \N$ such that $\Theta-\Xi=\{p^{n}\}$, $\Xi-\Theta=\{q^{m}\}$, and $\dist_{\Hor{\C{\down{\Delta_v}}}}(p^n,q^m)=1$; 
\item\label{hatW-edge_induced} \emph{Edge of $\W$-type}: $\Theta^+=\Xi^+$, and $\down{\Theta}$ is $\W$-adjacent to $\down{\Xi}$.
\end{enumerate}
The \emph{cusped space} for $(X,\W)$ is the pair $(\hat X,\hat \W)$. To avoid ambiguity, we denote augmented links in $(\hat X, \hat \W)$ by $\hat C(\cdot)$.
\end{defn}

\begin{rem}
Edges of cusp-type will make sure a copy of the combinatorial horoball on every minimal domain embeds inside $\hat\W$. The purpose of the edges of $\W$-type is twofold. On the one hand, they shall ensure $\W$ embeds as a subspace of $\hat \W$; on the other hand, the horoballs will be ``well separated", meaning that if a point of $\hat\W$ projects deep in some horoball then it must project to the base of any horoball on a transverse domain.
\end{rem}

\begin{rem}\label{rem:depth_diff}
   If $\Theta,\Xi\in \hat \W^{(0)}$ are $\hat \W$-adjacent, then $|\depth{\Theta}-\depth{\Xi}|\le 1$ by construction.
\end{rem}

\subsection{Properties of the cusped space}
We summarise the crucial properties of cusped spaces in the following theorem, most importantly that cusped spaces of relative \squid{} HHSs are \squid{} HHSs themselves.

\begin{thm}\label{thm:teich_is_hhs}
Let $(X,\W, n,\delta)$ be a relative \squid{} HHS, and let $(\hat X,\hat \W)$ be its cusped space. Then the following hold:
\begin{enumerate}[label=(\Alph*)]
\item\label{item_coarseemb} The inclusion $X\hookrightarrow \hat X$ induces an injective coarse embedding $i\colon\W\hookrightarrow \hat \W$. Moreover, the coarse embedding functions do not depend on any of the data.
\item\label{item_Teich_is_combhhs} There exists $\delta'$, only depending on $\delta$, such that $(\hat X,\hat \W, n,\delta')$ is a squid HHS.
\item\label{item_G-action_on_teich} Any action of a group $G$ on $(X,\W)$ extends to a $G$-action on $(\hat X,\hat \W)$.
\item\label{item_W_is_relHHS} $\W$ is a hierarchically hyperbolic space relative to $\{\C{\Delta_v}\}_{v\in \ov X^{(0)}}$.
\end{enumerate}
\end{thm}

\noindent \Cref{item_W_is_relHHS}, combined with \Cref{defn:rel_squid_HHG}, readily implies the following:
\begin{cor}
    A relative \squid{} HHG is a relative hierarchically hyperbolic group.
\end{cor}

\noindent We break the proof of \Cref{thm:teich_is_hhs} into a series of propositions. Firstly, we prove a quantitative version of \Cref{thm:teich_is_hhs}.\ref{item_coarseemb}:
\begin{prop}\label{prop_coarseemb}
\ref{item_coarseemb} Let $(X,\W)$ be a relative \squid{} HHS. The inclusion $X\hookrightarrow \hat X$ induces an injective coarse embedding $\W\hookrightarrow \hat \W$. More precisely, for every two maximal simplices $\Theta,\Xi\in \W^{(0)}$, 
$$m(\dist_{\W}(\Theta,\Xi))\le \dist_{\hat \W}(\Theta,\Xi)\le \dist_{\W}(\Theta,\Xi),$$
where $m$ is the inverse\footnote{This function is known as the Lambert $W$ function, and diverges at infinity. See e.g. \href{https://mathworld.wolfram.com/LambertW-Function.html}{https://mathworld.wolfram.com/LambertW-Function.html} for more details.} of the function $n\to n2^n$.
\end{prop}

\begin{proof}
    Let $\Theta$ be a maximal simplex of $X$, which we can also see as a maximal simplex of $\hat X$ with $\Theta^+=\emptyset$ and $\down{\Theta}=\Theta$. Thus there is an injection $\W\hookrightarrow \hat \W$, which maps $\W$-edges to $\hat \W$-edges of $\W$-type. Since $\W$ is an induced subgraph of $\hat \W$, the inclusion $\W\hookrightarrow\hat \W$ is $1$-Lipschitz. 
    
    We are left to prove that, if $\Theta,\Xi\in \W^{(0)}$ and $D=\dist_{\hat \W}(\Theta,\Xi)$, then $\dist_{\W}(\Theta, \Xi)\le D2^D.$ To this extent, let $\{\Sigma_0=\Theta, \Sigma_1,\ldots, \Sigma_D=\Xi\}$ be a $\hat \W$-geodesic, and consider the path $\{\down{\Sigma_0}=\Theta, \down{\Sigma_1},\ldots, \down{\Sigma_D}=\Xi\}$ in $\W$.
    
    \begin{claim}\label{claim:coarseemb}
        For every $i=0,\ldots, D-1$, $\dist_{W}(\down{\Sigma_i}, \down{\Sigma_{i+1}})\le 2^D$.
    \end{claim}

    \begin{claimproof}[Proof of \Cref{claim:coarseemb}]
        If $\Sigma_i$ and $\Sigma_{i+1}$ span an edge of $\W$-type then $\down{\Sigma_i}$ and $\down{\Sigma_{i+1}}$ are $\W$-adjacent, and therefore  $\dist_{W}(\down{\Sigma_i}, \down{\Sigma_{i+1}})\le 1$. Thus suppose that $\Sigma_i$ and $\Sigma_{i+1}$ are joined by an edge of cusp-type, so that there exist $v\in \ov X^{(0)}$ and $\Hor{\C{\down{\Delta_v}}}$-adjacent points $p^n,q^m\in L_v\times \N$ such that $\Sigma_i-\Sigma_{i+1}=\{p^n\}$ and $\Sigma_{i+1}-\Sigma_{i}=\{q^m\}$. By inspection of \Cref{def:horoball} there are two sub-cases to analyse. If $q^m=p^{n+1}$ then $\down{\Sigma_{i}}= \down{\Sigma_{i+1}}$, and we are done. Otherwise we must have that $m=n$ and $\dist_{\C{\Delta_v}}(p,q)\le 2^n$, so there is a $\C{\Delta_v}$-geodesic with vertices $p_0=p,p_1,\ldots,p_k=q$ for some $k\le 2^n$. Since $L_v\times \N=\link{X}{\Sigma_i\cap\Sigma_{i+1}}$, \Cref{CHHS_def_edge_in_link} for $(X,\W)$ then implies that $\Sigma_i$ and $\Sigma_{i+1}$ are at distance at most $k\le 2^n$ in $\W$. But now $n\le \depth{\Sigma_i}$, which by \Cref{rem:depth_diff} is at most $D+\depth{\Sigma_0}=D$, as required.
    \end{claimproof}
    \noindent \Cref{prop_coarseemb} now follows from the Claim, since 
    \begin{align*}\dist_{W}(\Theta, \Xi)\le \sum_{i+0}^{D-1}\dist_{W}(\Sigma_i, \Sigma_{i+1})\le D2^D. & \qedhere \end{align*}
\end{proof}

\noindent Before proceeding with the proof of \Cref{thm:teich_is_hhs}, we point out the following:
\begin{lem}\label{lem:embedding_of_duaug}
    The inclusion $X\hookrightarrow\hat X$ induces an embedding $\duaug{X}{\W}\to \duaug{\hat X}{\hat \W}$.
\end{lem}

\begin{proof} Firstly, as mentioned at the beginning of the proof of \Cref{prop_coarseemb}, maximal simplices of $X$ are also maximal simplices of $\hat X$. Moreover, if $p,q\in X$ belong to $\W$-adjacent maximal simplices $\Phi,\Psi$ of $X$, then $\Phi$ and $\Psi$ are joined by an edge of $\W$-type when seen as simplices of $\hat X$ (recall \Cref{defn:cusped_space}.\ref{hatW-edge_induced}). This proves that the inclusion $X\hookrightarrow \hat X$ induces a simplicial map $\duaug{X}{\W}\to \duaug{\hat X}{\hat \W}$, which is injective at the level of vertices.
    
    To prove that the above map is an embedding, we are left to show that, if two vertices $p,q$ belong to $\hat \W$-adjacent maximal simplices $\Sigma,\Theta$ of $\hat X$, respectively, then there exist $\W$-adjacent maximal simplices of $X$ containing $p$ and $q$, respectively. 
    
    If $\Sigma$ and $\Theta$ are joined by an edge of cusp-type, then $\Sigma-\Theta=\{p\}$, $\Theta-\Sigma=\{q\}$, and $\dist_{\Hor{\C{\Delta_{\p(p)}}}}(p,q)=1$. Since both $p$ and $q$ belong to $X$, this means that $\dist_{{\C{\Delta_{\p(p)}}}}(p,q)=1$, so $p$ and $q$ belong to $\W$-adjacent maximal simplices of $X$.

    Otherwise $\Sigma$ and $\Theta$ are joined by an edge of $\W$-type, which means that $\down{\Sigma}$ and $\down{\Theta}$ are $\W$-adjacent. Notice moreover that $p\in \down{\Sigma}$ and $q\in \down{\Theta}$ because they lied in $X$, and we are done.
\end{proof}

\noindent We now prove that the cusped space is a \squid{} HHS:
\begin{prop}\label{prop:teich_is_comb}
\ref{item_Teich_is_combhhs} Let $(X,\W)$ be a relative \squid{} HHS. There exists $\delta'$, only depending on $\delta$, such that $(\hat X,\hat \W, n,\delta')$ is a \squid{} HHS.
\end{prop}

\begin{proof} In what follows, we say that a constant is \emph{uniform} if it only depends on $n$ and $\delta$. We check that each axiom from \Cref{def:combHHS} holds for some uniform constant $\delta_i$, and we shall then set $\delta'=\max_i\delta_i$. 

\Cref{CHHS_def_complexity} follows from \Cref{cor:nesting_of_links_in_cusp} and the corresponding axiom for $(X,\W)$, with the same constant $n$. Regarding \Cref{CHHS_def_hyperbolicity}, by the description of links in a blowup from \Cref{cor:bounded_links}, it is enough to restrict to simplices of either blowup type or cone type. We analyse the two cases separately in \Cref{claim:hyp2} and \Cref{claim:hyp3} below.
\begin{claim}[Blowup type]\label{claim:hyp2}
Let $\Sigma$ be a simplex of blowup type. There exists a $2$-quasi-isometry $\C{\down{\Sigma}}\to \hatC{\Sigma}$. In particular, $\hatC{\Sigma}$ is $\delta_1$-hyperbolic for some $\delta_1$ only depending on $\delta$.
\end{claim}
\begin{claimproof}[Proof of \Cref{claim:hyp2}]
    Let $\duaug{\link{\ov X}{\ov\Sigma}}{\W}$ be the graph obtained from $\link{\ov X}{\ov\Sigma}$ by connecting two vertices whenever they belong to $\W$-adjacent simplices, and define $\duaug{\link{\ov X}{\ov\Sigma}}{\hat\W}$ analogously. By how we defined $\hat\W$-edges in \Cref{defn:cusped_space}, we see that two vertices in $\link{\ov X}{\ov\Sigma}$ are $\W$-adjacent if and only if they are $\hat\W$-adjacent, so $\duaug{\link{\ov X}{\ov\Sigma}}{\W}$ and $\duaug{\link{\ov X}{\ov\Sigma}}{\hat\W}$ coincide. Now, since $X$ is a blowup of $\ov X$ and $\Sigma$ is of blowup type, the retraction $\p$ induces a quasi-isometry $\C{\down{\Sigma}}\to \duaug{\link{\ov X}{\ov\Sigma}}{\W}$, and similarly the inclusion $\duaug{\link{\ov X}{\ov\Sigma}}{\W}\to \hatC{\Sigma}$ is an isometric, $2$-surjective embedding. By keeping track of the quasi-isometry constants, one sees that the composition $\C{\down{\Sigma}}\to \hatC{\Sigma}$ is a $2$-quasi-isometry, as required. 
\end{claimproof}
\begin{claim}[Cone type]\label{claim:hyp3} 
Let $\Delta_v$ be a simplex of cone type, for some $v\in \ov X^{(0)}$. Then $\hatC{\Delta_v}\cong\Hor{\C{\down{\Delta_v}}}$, and is therefore $\delta_0$ hyperbolic by \Cref{thm:horgamma_hyp}.
\end{claim}
\begin{claimproof}[Proof of \Cref{claim:hyp3}]
    The vertex set of $\hatC{\Delta_v}$ is $L_v\times \N$. By inspection of \Cref{defn:cusped_space}, two distinct vertices $p^n,q^m\in L_v\times \N$ are adjacent in $\Hor{\C{\down{\Delta_v}}}$ if and only if they are connected in $\hatC{\Delta_v}$ by an edge of cusp-type. Hence we are left to prove that, if $p^n\in\Theta$, $q^n\in \Xi$, where $\Theta,\Xi\in \hat \W$ span an edge of $\W$-type, then $p^n$ and $q^n$ were already joined by an edge of cusp-type. Indeed, since $\Theta^+=\Xi^+$ and $p^n$ is not equal to $q^m$, we must have that $n=m=0$. But then $p^0\in \down{\Theta}$ and $q^0\in \down{\Xi}$ are $\W$-adjacent, meaning that $\dist_{\C{\down{\Delta_v}}}(p,q)=1$ and therefore $p^0,q^0$ were already joined by an edge of cusp-type.
\end{claimproof}

\noindent \Cref{CHHS_def_qi_emb} is also proven by considering two cases, according to \Cref{cor:bounded_links}:
\begin{claim}[Blowup type]\label{claim:blowuptype}
There exists $\delta_2$, only depending on $\delta$, such that, if $\Sigma$ is a simplex of blowup type, then $\hatC{\Sigma}$ $\delta_2$-quasi-isometrically embeds inside $Y_\Sigma$.
\end{claim}
\begin{claimproof}[Proof of \Cref{claim:blowuptype}]
    Notice first that every $\hat X$-simplex with the same link as $\Sigma$ must itself be of edge type. This implies that $Y_\Sigma$ is the $\p$-preimage of $\p(Y_\Sigma)$ in $\hat X$, so $\p$ restricts to a $2$-quasi-isometry $Y_\Sigma\to \p(Y_\Sigma)$. The same arguments apply to $\down{\Sigma}$ inside $X$, and notice that $\p(Y_\Sigma)=\p(Y_{\down{\Sigma}})$. Then the $\delta$-quasi-isometric embedding $\C{\down{\Sigma}}\to Y_{\down{\Sigma}}$ induces a uniform quasi-isometric embedding $$\duaug{\link{\ov X}{\ov \Sigma}}{\hat \W}=\duaug{\link{\ov X}{\ov \Sigma}}{\W}\to \p(Y_{\down{\Sigma}})=\p(Y_\Sigma),$$ with constants only depending on $\delta$, and in turn this induces a $\delta_2$-quasi-isometric embedding $\hatC{\Sigma}\to Y_{\Sigma}$ for some $\delta_2$ only depending on $\delta$.
    \end{claimproof}
\begin{claim}[Cone type]
\label{claim:cone-type}
     There exists $\delta_3$, only depending on $\delta$, such that, if $\Delta_v$ is a simplex of cone type, for some $v\in \ov X^{(0)}$, then $\hatC{\Delta_v}$ $\delta_3$-quasi-isometrically embeds inside $Y_{\Delta_v}$.
\end{claim}
\begin{claimproof}[Proof of \Cref{claim:cone-type}]
    The fact that $\Delta_v$ is of cone type implies that $Y_{\Delta_v}$ is spanned by 
    $$(L_v\times \N) \cup \bigcup_{w\in \ov X-\CStar{v}} \{w\}\star(L_w\times\N).$$
    $Y_{\Delta_v}$ uniformly retracts onto $\ov Y_{\Delta_v}$, defined as the subgraph of $\duaug{\hat X}{\hat \W}$ spanned by 
    $$(L_v\times \N) \cup (\ov X-\CStar{v}).$$
    Let $\ov Y_{\down{\Delta_v}}$ be the subgraph of $\duaug{X}{\W}$ (hence of $\duaug{\hat X}{\hat \W}$ by \Cref{lem:embedding_of_duaug}) spanned by 
    $$L_v\times\{0\} \cup (\ov X-\CStar{v}).$$
    Notice that, if $p^n\in L_v\times \N$ and $w\in \ov X-\CStar{v}$ are $\hat \W$-adjacent, then $n=0$, by how we defined $\hat \W$-edges in \Cref{defn:cusped_space}. Hence
    \begin{equation}\label{eq:dec_of_Ydelta}
        \ov Y_{\Delta_v}=\hatC{\Delta_v}\cup_{\C{\down{\Delta_v}}} \ov Y_{\down{\Delta_v}}.
    \end{equation}
    We shall now prove that $\hatC{\Delta_v}$ is uniformly quasi-isometrically embedded in $\ov Y_{\Delta_v}$, and therefore in $Y_{\Delta_v}$. Fix any $\ov Y_{\Delta_v}$-geodesic $\gamma$ connecting two points $a,b\in L_v\times \N$, and we shall prove that its length $\ell(\gamma)$ is bounded below by a linear function of $\dist_{\hatC{\Delta_v}}(a,b)$. By \Cref{eq:dec_of_Ydelta}, $\gamma$ can be decomposed as a concatenation $\eta_1*\lambda_1*\ldots*\eta_k$, where each $\eta_i$ is a geodesic in $\hatC{\Delta_v}$ and each $\lambda_j$ is a geodesic in $\ov Y_{\down{\Delta_v}}$ with endpoints $x_j^0,y_j^0\in \C{\down{\Delta_v}}$. Furthermore, by \Cref{CHHS_def_qi_emb} applied to $(X,\W)$, there is a constant $D=D(\delta)$ such that $\C{\down{\Delta_v}}$ is $D$-quasi-isometrically embedded in $\ov Y_{\down{\Delta_v}}$. Then
    $$\dist_{\hatC{\Delta_v}}(x_j^0,y_j^0)\le \dist_{\C{\down{\Delta_v}}}(x_j^0,y_j^0)\le D\ell(\lambda_i)+D.$$
    In turn, this means that
        \begin{align*}
        \dist_{\hatC{\Delta_v}}(a,b)&\le \sum_{i=1}^k \ell(\eta_i) +\sum_{i=1}^{k-1} \dist_{\hatC{\Delta_v}}(x_j^0,y_j^0)\\
        &\le \sum_{i=1}^k \ell(\eta_i) +D\sum_{i=1}^{k-1} \ell(\lambda_i) +(k-1)D\\
        &\le D\left(\sum_{i=1}^k \ell(\eta_i) +\sum_{i=1}^{k-1} \ell(\lambda_i)\right)+\ell(\gamma)D=2D\ell(\gamma).\qedhere\\
    \end{align*}
\end{claimproof}
\noindent Moving to the remaining items of \Cref{def:combHHS}, the fact that $\hat X$ has \cleanish{}, which follows from \Cref{lem:cleanish_for_blowup}, implies \Cref{CHHS_def_lk_cap_lk} as in the proof of \cite[Theorem 6.4]{BHMS} (more precisely, at the beginning of the paragraph named “\emph{$(X,W)$ is a combinatorial HHS}”).

We finally address \Cref{CHHS_def_edge_in_link}. Let $\Theta$ be an $\hat X$-simplex and $x,y\in \link{\hat X}{\Theta}$ be distinct vertices which belong to $\W$-adjacent maximal simplices $\Sigma_x$ and $\Sigma_y$. If the latter span an edge of cusp-type then there exists $v\in \ov X^{(0)}$ such that $x,y\in L_v\times \N$ and $\dist_{\Hor{\C{\Delta_v}}}(x,y)=1$. Since $x,y\in \link{\hat X}{\Theta}$, we must have that $v\in \CStar{\ov\Theta}$, so $\Theta$ can be completed to two maximal simplices connected by an edge of cusp-type and containing $x$ and $y$, respectively. Suppose instead that $\Sigma_x$ and $\Sigma_y$ span an edge of $\W$-type. Then $\Sigma_x^+=\Sigma_y^+$, and since $x,y$ are distinct we must have that $x,y\in X$, so $x,y\in \link{X}{\down{\Theta}}$ by \Cref{lem:links_in_cusped}. By \Cref{CHHS_def_edge_in_link} applied to $(X,\W)$, there exist two $\W$-adjacent $X$-maximal simplices $\Pi_x\supseteq x\star \down{\Theta}$ and $\Pi_y\supseteq x\star \down{\Theta}$. If we call $\Psi_x$ the $\hat X$-maximal simplex such that $\Psi_x^+=\Theta^+$ ad $\down{\Psi_x}=\Pi_x$, and we define $\Psi_y$ analogously, then $\Psi_x\supseteq x\star \Theta$, $\Psi_y\supseteq y\star \Theta$, and $\Psi_x,\Psi_y$ span an edge of $\W$-type. This concludes the proof of \Cref{CHHS_def_edge_in_link}, and in turn of \Cref{prop:teich_is_comb}.
\end{proof}

\begin{rem}
    If $(X,\W)$ was a genuine \squid{} HHS, not just a relative one, another approach to proving \Cref{claim:cone-type} would be to use that $\ov Y_{\down{\Delta_v}}$ is hyperbolic by \cite[Proposition 3.3]{BHMS} (which is the key step in the proof that a combinatorial HHS is indeed a HHS). This would give a closest-point projection from $\ov Y_{\down{\Delta_v}}$ to its quasiconvex subset $\C{\down{\Delta_v}}$. For a relative \squid{} HHS we cannot use \cite{BHMS}; however, if $\ov Y_{\down{\Delta_v}}$ was hyperbolic relative to $\C{\down{\Delta_v}}$, we could still use that closest-point projections to peripheral subsets of relatively hyperbolic spaces are coarsely Lipschitz (see \cite{Sisto_proj_relhyp} for general results about these projections). 
    
    It can in fact be extracted a posteriori from our arguments and \cite[Proposition 3.3]{BHMS} that relative hyperbolicity indeed holds. Indeed, $Y_{\Delta_v}$ is quasi-isometric to $\ov Y_{\down{\Delta_v}}$ with a combinatorial horoball glued onto $\C{\down{\Delta_v}}$, and via \cite[Proposition 3.3]{BHMS} we know that $Y_{\Delta_v}$ is hyperbolic. This implies that $\ov Y_{\down{\Delta_v}}$ is hyperbolic relative to $\C{\down{\Delta_v}}$, as this is one of the characterisations of relative hyperbolicity. 
\end{rem}

\begin{prop}\ref{item_G-action_on_teich}
    Let $G$ be a group acting on a relative \squid{} HHS $(X,\W)$. Then the action extends to a $G$-action on $(\hat X,\hat \W)$.
\end{prop}

\begin{proof}
    For every $g\in G$, $v\in \ov X^{(0)}$, and $p^n\in L_v\times \N$, set $g\cdot p^n\coloneq (g\cdot p)^n\in L_{g\cdot v}\times \N$, where $g\cdot v$ is the image of $v$ under the action on $\ov X$, seen as a subgraph of $X$. By inspection of \Cref{defn:cusped_space}, this induces an action on $\hat \W$; furthermore,  $G$ still acts on $\ov X$ with finitely many orbits of simplices, and therefore we just defined a $G$-action on $(\hat X,\hat \W)$.
\end{proof}

\begin{prop}\label{prop:W_is_relHHS} 
\ref{item_W_is_relHHS} Let $(X,\W)$ be a relative \squid{} HHS. Then $\W$ is a hierarchically hyperbolic space relative to $\{\C{\Delta_v}\}_{v\in \ov X^{(0)}}$.
\end{prop}
\begin{proof}
As in \Cref{defn:simplex_equivalence}, let $\frakS$ be the collection of $\sim$-equivalence classes of non-maximal simplices of $X$, with the relations $\nest$ and $\orth$ as in \Cref{defn:nest_orth}. It follows from \Cref{lem:links_in_cusped} that the relations $\sim$, $\nest$, and $\orth$ are the same if we consider links in $X$ or in $\hat X$. For every $[\Sigma]\in\frakS$ let $\C{\Sigma}$ and $\hatC{\Sigma}$ be the associated augmented links in $(X,\W)$ and in $(\hat X, \hat \W)$, respectively. 

For every $[\Sigma]\in\frakS$, let $\pi_{[\Sigma]}\colon \hat \W\to \hatC{\Sigma}$ be the projection from \Cref{defn:projections}. If $\Sigma=\Delta_v$ for some $v\in\ov X^{(0)}$, then \Cref{eq:dec_of_Ydelta} shows that $\C{\Delta_v}$ separates $\hatC{\Delta_v}$ from the rest of $Y_{\Delta_v}$. As a consequence, given $w\in \W$, its projection $\pi_{[\Delta_v]}(w)$ must intersect $ \C{\Delta_v}$, since it is defined as the coarse closest point projection of $w\cap Y_{\Delta_v}$ to $\hatC{\Delta_v}$. We then define $\down{\pi}_{[\Delta_v]}\colon \W\to \C{\Delta_v}$ by setting $\down{\pi}_{[\Delta_v]}(w)={\pi}_{[\Delta_v]}(w)\cap \C{\Delta_v}$. 

If instead $\Sigma$ is of edge-type, $\hatC{\Sigma}$ and $\C{\Sigma}$ are uniformly quasi-isometric, so we can always replace one with the other (in the presence of a group action, we can also choose the quasi-isometry to be $\Stab{G}{\Sigma}$-equivariant). Therefore we define $\down{\pi}_{[\Sigma]}\colon \W\to \C{\Sigma}$ as the composition of ${\pi}_{[\Sigma]}$, restricted to $\W$, with the fixed identification $\hatC{\Sigma}\to\C{\Sigma}$.

We now prove that the data $(\W,\frakS, \C{\cdot},\down{\pi})$ make $\W$ a relative hierarchically hyperbolic space, by checking the axioms of \cite[Definition~2.8]{Russell_Rel_Hyp} (which is equivalent to the original formulation from \cite{HHS_II}). We shall often use without reference that, since $(\hat X, \hat \W)$ is a combinatorial HHS by \Cref{prop:teich_is_comb}, then $(\hat\W,\frakS, \hatC{\cdot},\pi)$ is a hierarchically hyperbolic space by \cite[Theorem 1.18]{BHMS}, and therefore satisfies all axioms.

All requirements only involving $\nest$ and $\orth$, which only depend on $X$, follow exactly as in the proof of \cite[Theorem 1.18]{BHMS}, so we focus on those involving projections and coordinate spaces.

\begin{enumerate}[label=(\arabic*{})]
        \item\label{axiom:projections} \textbf{(Projections.)} 
        For every $[\Sigma]\in\frakS$, the coarse map $\down{\pi}_{[\Sigma]}\colon \W\to \C{\Sigma}$ is surjective since every vertex of $\C\Sigma$ can be completed to a maximal simplex of $X$. If $\Sigma$ is not of the form $\Delta_v$ for any $v\in \ov X^{(0)}$, then $\down{\pi}_{[\Sigma]}$ is uniformly coarsely Lipschitz, as it is the composition of the coarsely Lipschitz map $\pi_{[\Sigma]}\colon \W\to \hatC{\Sigma}$ with the quasi-isometry $\hatC{\Sigma}\to \C{\Sigma}$. If instead $\Sigma=\Delta_v$ and $w,w'\in \W$ are $\W$-adjacent, then $\dist_{\hatC{\Delta_v}}(\pi_{[\Delta_v]}(w),\pi_{[\Delta_v]}(w'))$ is uniformly bounded; hence $\dist_{\C{\Delta_v}}(\down{\pi}_{[\Delta_v]}(w),\down{\pi}_{[\Delta_v]}(w'))$ is uniformly bounded as well, since $\C{\Delta_v}$ is uniformly coarsely embedded in $\hatC{\Delta_v}$ by \Cref{lem:coarseemb_in_hor}.

        \item\label{axiom:nesting} \textbf{(Nesting and transversality.)} 
        Whenever $[\Sigma],[\Theta]\in\frakS$ are such that $[\Sigma]\propnest [\Theta]$, there exists a uniformly bounded subset $\rho^{[\Sigma]}_{[\Theta]}\in\hatC{\Theta}$, defined as in \Cref{defn:projections}. Since $[\Theta]$ is not $\nest$-minimal, $\hatC{\Theta}$ and $\C{\Theta}$ are quasi-isometric, so $\rho^{[\Sigma]}_{[\Theta]}$ is uniformly bounded in $\C{\Theta}$ as well.

        Similarly, if $[\Sigma]\transverse [\Theta]\in\frakS$ and $\Theta$ is not of the form $\Delta_v$ for any $v\in \ov X^{(0)}$, then $\rho^{[\Sigma]}_{[\Theta]}$ is uniformly bounded in $\C{\Theta}$ as well. If instead $\Theta=\Delta_v$ then $\rho^{[\Sigma]}_{[\Delta_v]}\in \C{\Delta_v}$, as it was defined using the coarse closest point projection $Y_{\Delta_v}\to \hatC{\Delta_v}$. Hence $\rho^{[\Sigma]}_{[\Delta_v]}$ is uniformly bounded in $\C{\Delta_v}$, as the latter is uniformly coarsely embedded in $\hatC{\Delta_v}$.

        \item\label{axiom:consistency} \textbf{(Consistency.)} 
        The consistency inequalities for $(\hat X,\hat \W)$ impose uniform bounds involving certain $\pi$-projections in $\hatC{\cdot}$ and the relative projections between domains. As above, one can deduce analogous uniform bounds involving the corresponding $\down{\pi}$-projections in $\C{\cdot}$, using the uniform coarse embeddings $\C{\cdot}\to \hatC{\cdot}$.
		
        \item\label{axiom:hyperbolicity} \textbf{(Hyperbolicity of non-minimal domains)} Whenever $\Sigma\neq \Delta_v$ for any $v\in \ov X^{(0)}$, $\C{\Sigma}$ is uniformly quasi-isometric to the $\delta'$-hyperbolic space $\hatC{\Sigma}$. 
		
        \item\label{axiom:bounded_geodesic_image} \textbf{(Bounded geodesic image.)} 
        Let $w,w'\in \W$ and  $[\Sigma],[\Theta]\in\frakS$ be such that $[\Sigma]\propnest [\Theta]$. Since $\C{\Sigma}$ is uniformly coarsely embedded in $\hatC{\Sigma}$, there exists a constant $B=B(\delta')$ such that, if $\dist_{\C{\Sigma}}(\down{\pi}_{[\Sigma]}(w),\down{\pi}_{[\Sigma]}(w')) > B$, then $\dist_{\hatC{\Sigma}}(\pi_{[\Sigma]}(w),\pi_{[\Sigma]}(w')) > \delta'$. By the bounded geodesic image axiom for $(\hat X, \hat \W)$, every $\hatC{\Theta}$--geodesic from $\pi_{[\Theta]}(w)$ to $\pi_{[\Theta]}(w)$ must therefore pass $\delta'$-close to $\rho^{[\Sigma]}_{[\Theta]}$.

        Now let $\gamma$ be a $\C{\Theta}$-geodesic from $\down{\pi}_{[\Theta]}(w)$ to $\down{\pi}_{[\Theta]}(w)$, and let $\hat\gamma$ be its image under the quasi-isometry $\hatC{\Theta}\to \C{\Theta}$ (here we are using that $[\Theta]$ is not $\nest$-minimal). Then $\hat\gamma$ is a uniform quality quasi-geodesic, which is within uniform Hausdorff distance from a geodesic by the Morse lemma for hyperbolic spaces Morse lemma (see, e.g., \cite[III.H.1.7]{BH}). Hence $\gamma$ must pass uniformly close to $\rho^{[\Sigma]}_{[\Theta]}$, thus proving the bounded geodesic image property for $(X,\W)$.

        \item\label{axiom:partial_realisation} \textbf{(Partial realization.)}
        Let $\{[\Sigma_i]\}$ be a finite collection of pairwise orthogonal domains, together with a choice of points $p_i\in \C{\Sigma_i}^{(0)}$. Since $\link{X}{\Sigma_i}\subseteq \link{X}{\link{X}{\Sigma_j}}$ for all $i\neq j$, the points $p_i$ span a simplex in $X$, so they belong to a maximal simplex $w\in \W^{(0)}$. As argued in the proof of \cite[Theorem 1.18]{BHMS} (see \cite[page 43]{BHMS}), $w$ is a realisation point for the collection in $(\hat X,\hat \W)$, and therefore in $(X,\W)$ since all coordinates $p_i$ were taken in $\C{\Sigma_i}$.

        \item\label{axiom:large_link_lemma} \textbf{(Large links.)} By the arguments from \cite[Section 4.8]{Durham_Cubulating_Infinity}, it is enough to check the \emph{passing up axiom}, as defined in e.g. \cite[Definition 4.1.(12')]{randomquot_hhg}. The passing up axiom for $(\hat X,\hat \W)$ provides the existence of a function $P\colon \R_{\ge 0}\to \R_{\ge 0}$ with the following property. Let $t\ge 0$, $[\Sigma]\in\frakS$, and $w,w'\in \W$. If $\dist_{\hatC{\Theta_i}}(w,w')>\delta'$ for a collection of domains $\{[\Theta_i]\}_{i=1}^{P(t)}\subseteq \frakS$ nested in $[\Sigma]$, then there exists a non-maximal simplex $\Xi$ of $\hat X$ such that $[\Xi]\nest[\Sigma]$, $[\Xi]$ properly contains some $[\Theta_i]$, and $\dist_{\hatC{\Xi}}(w,w')\ge t$. Since $\link{\hat X}{\Xi}=\link{\hat X}{\down{\Xi}}$, up to replacing $\Xi$ by $\down{\Xi}$ we can assume that $\Xi\subseteq X$.
    
        Now let $E\ge 0$ be such that, for every $[\Theta]\in\frakS$,  if $\dist_{\C{\Theta}}(w,w')>E$, then $\dist_{\hatC{\Theta}}(w,w')> \delta'$, which exists as $\C\Theta$ is uniformly coarsely embedded in $\hatC{\Theta}$. Hence, if $\dist_{\C{\Theta_i}}(w,w')>E$ for a collection $\{[\Theta_i]\}_{i=1}^{P(t)}$ as above, then the passing up axiom for $(\hat X, \hat \W)$ produces a simplex $\Xi$ as above, such that $\dist_{\C{\Xi}}(w,w')\ge\dist_{\hatC{\Xi}}(w,w')\ge t$. This proves that the passing up axiom is satisfied by $(X, \W)$, with the same function $P$ and with $\delta'$ replaced by $E$.

        \item\label{axiom:uniqueness} \textbf{(Uniqueness.)}
        If $w,w'\in \W^{(0)}$ and $r\in \mathbb{R}_{\ge 0}$ are such that $\dist_{\C{\Sigma}}(w,w')\le r$ for every $[\Sigma]\in \frakS$, then $\dist_{\hatC{\Sigma}}(w,w')\le \dist_{\C{\Sigma}}(w,w')\le r$. Then the uniqueness axiom for $(\hat X,\hat \W)$ yields that $\dist_{\hat \W}(w,w')\le \theta(r)$ for some function $\theta\colon \mathbb{R}_{\ge 0}\to \mathbb{R}_{\ge 0}$. In turn, since $\W\hookrightarrow \W$ is a coarse embedding by \Cref{item_coarseemb}, $\dist_{\W}(w,w')$ is also bounded in terms of $r$.\qedhere
    \end{enumerate}
\end{proof}

\section{Some examples}
\label{sec:examples}
We now present some constructions of relative \squid{} HHSs and HHGs, and describe the associated cusped spaces; further examples (namely \emph{short HHG}) will be discussed in \Cref{sec:short_are_squid} below. 

 \subsection{Relatively hyperbolic groups}
    Here we first construct a relative \squid{} HHG structure for relatively hyperbolic groups (subject to natural finiteness conditions); then we recognise the associated cusped space as the \emph{augmented space}, in the sense of e.g. \cite[Definition 4.3]{Hruska}.

    \begin{lem}\label{lem:relhyp_is_squid}
        A finitely generated group $G$, which is hyperbolic relative to a finite collection $\mathcal P$ of finitely generated subgroups, is a relative \squid{} HHG.
    \end{lem}
    \begin{proof} Let $S$ be a finite generating set for $G$, such that $S\cap P$ generates $P$ for all $P\in \mathcal P$, and let $\dist_S$ be the associated word metric on $G$. We shall prove that $G$ has a relative \squid{} HHG structure $(X,\W)$, where:
        \begin{itemize}
            \item $\ov X$ is the discrete graph whose vertices are cosets of the subgroups in $\mathcal P$;
            \item for every $g\in G$ and $P\in \mathcal P$, $L_{gP}=gP^{(0)}$;
            \item For every $g,g'\in G$, $P,P'\in\mathcal P$, $x\in gP$ and $x'\in g'P'$, the maximal simplices $(gP,x)$ and $(g'P',x')$ span a $\W$-edge if and only if $\dist_S(x,x')\le 1$.
        \end{itemize}
        It is clear that $\ov X$ has \cleanish{}. We now check the axioms from \Cref{def:combHHS}. \Cref{CHHS_def_complexity}  follows from \Cref{lem:decomposition_of_links} and the fact that $\ov X$ is discrete. For \Cref{CHHS_def_hyperbolicity}, $\C{\emptyset}=\duaug{X}{W}$ is isometric to the coned-off graph, which is obtained from the Cayley graph $\Cay{G}{S}$ by adding a cone over each coset of the peripherals, and is hyperbolic by e.g. \cite[Theorem 5.1]{Hruska}. We also notice that, for every $gP\in G/P$, $\C{\Delta_{gP}}\cong g\Cay{P}{S\cap P}$. 
        
        Moving to \Cref{CHHS_def_qi_emb}, $Y_{\Delta_{gP}}=X-\{gP\}$ retracts onto $\C{\Delta_{gP}}$ since there is a coarse closest point projection $G\to gP$, which map other cosets of the peripherals to uniformly bounded sets (see e.g. \cite[Theorem 2.14]{Sisto_proj_relhyp} and the arguments therein). \Cref{CHHS_def_lk_cap_lk} again follows from the fact that $\ov X$ has \cleanish{}. Finally, if $x,y\in X$ belong to the link of a simplex $\Sigma$, then $\Sigma=\{gP\}$ for some $P\in \mathcal P$ and $g\in G$, and \Cref{CHHS_def_edge_in_link} follows.

        We are left to check that the above defines a relative \squid{} \emph{HHG} structure for $G$. This is because $G$ acts cofinitely on the collection of cosets of peripherals; moreover, the action on $X$ naturally extends to an action on $\W$, and the latter is isometric to $\Cay{G}{S}$ by construction.
    \end{proof}

    \noindent Recall from e.g. \cite[Definition 4.3]{Hruska} that the \emph{augmented space} for $(G,\mathcal P)$ is the graph obtained from $\Cay{G}{S}$ by gluing a combinatorial horoball on $gP$ whenever $g\in G$ and $p\in \mathcal P$. 
    \begin{lem}
       In the setting of \Cref{lem:relhyp_is_squid}, the cusped space $\hat \W$ is $G$-equivariantly isometric to the augmented space.
    \end{lem}
    \begin{proof}
        This readily follows by how we constructed $\hat \W$-edges in \Cref{defn:cusped_space}.
    \end{proof}

\subsection{Most HHSs are \squid{}}\label{sec:converse_are_squid}
In \cite{converse}, together with Mark Hagen, we roughly proved that combinatorial HHSs are abundant among all HHSs: given a hierarchically hyperbolic space $\mathcal Z$, whose hierarchical structure $\frakS$ is subject to certain natural requirements, there is a combinatorial HHS $(X,\W)$ such that $\mathcal Z$ and $\W$ are quasi-isometric (see \cite[Theorem 3.15]{converse} for the statement and \cite[Sections 3.1 to 3.4]{converse} for the requirements). Notably, from the description of the combinatorial structure in \cite[Section 4]{converse}, one can extract that $(X,\W)$ is a \squid{} HHS, where:
\begin{itemize}
    \item the support graph $\ov X$ is the so-called \emph{minimal orthogonality graph}, whose vertices correspond to minimal domains in the structure and edges correspond to orthogonality;
    \item for a minimal domain $V\in \frakS$, seen as a vertex of $\ov X$, the augmented link of $\Delta_V$ is quasi-isometric to the coordinate space $\mathcal C V$. 
\end{itemize} In particular, if $G$ acts on $\mathcal Z$ geometrically, and there is a ``compatible" action on $\frakS$ with finitely many orbits of collections of pairwise orthogonal domains, then $G$ is a \squid{} HHG.

\begin{rem}
    As pointed out in \cite[Remark 3.17]{converse}, in the proof of \cite[Theorem 3.15]{converse}, hyperbolicity of coordinate spaces for minimal domains in $\frakS$ is only ever used to show that certain minimal augmented links are hyperbolic. In particular, the arguments there run verbatim to show that a \emph{relative} hierarchically hyperbolic space (resp. group), subject to the same assumptions on $\frakS$ (and on the group action on it), is a relative \squid{} HHS (resp. HHG).
\end{rem}

\subsubsection{Mapping class group and Teichm\"{u}ller space} Here we specialise the above construction to the mapping class group $\MCG(S)$ of a finite-type surface $S$, and recognise the associated cusped space as the Teichm\"{u}ller space of $S$. To do so, we must first delve into the various HHG structures for $\MCG(S)$, including the combinatorial one obtained from applying \cite[Theorem 3.15]{converse}.

By a \emph{surface} we mean a connected, compact, oriented 2-manifold $S$ with finitely many points removed. Given a subsurface $U$ of $S$ which is not a pair of pants, the associated \emph{curve graph} $\C U$ is the simplicial graph whose vertices are isotopy classes of simple closed curves on $U$, and adjacency corresponds to having disjoint representatives. There are two exceptions to this definition: if $U$ is a sphere with four punctures or a torus with a puncture, we define two curves to be adjacent if they have the minimal intersection number, which is two on the sphere and one on the torus; if instead $U=A_v$ is an annulus with core curve $v\in \C{S}^{(0)}$, we define its \emph{annular curve graph} as the quasiline from \cite[Section 2.4]{Masur_Minsky_2}.

As explained in \cite[Theorem G]{HHS_I}, the mapping class group $\MCG(S)$ admits a HHG structure $\frakS$ where domains correspond to subsurfaces (up to isotopy), and where the coordinate space for a subsurface is its curve graph (resp. annular curve graph if the subsurface is an annulus). Furthermore, \cite[Theorem 9.8]{converse} equips $\MCG(S)$ with a \squid{} HHG structure $(X,\W)$ with blowup data $(\ov X=\C S, \{\C{A_v}\}_{v\in \C S^{(0)}})$. By how $\W$-edges are defined in \cite[Definition 4.10]{converse}, unbounded coordinate spaces correspond to (annular) curve graphs: 
\begin{itemize}
    \item Let $v\in \C S^{(0)}$ be a curve. Then $\C{\Delta_v}$ and the annular curve graph $\C{A_v}$ have the same vertices, and the identity map is a quasi-isometry $\phi_{\Delta_v}\colon \C{\Delta_v}\to \C{A_v}$ with uniform constants. 
    \item Let $\ov \Sigma$ be a non-maximal simplex of $\C S$, and let $\Sigma\subseteq X$ be a simplex of blowup type whose support is $\ov \Sigma$. If we see $\ov\Sigma$ as a collection of disjoint curves on $S$, its complement minus the pants components is an open subsurface $U_\Sigma$. Then there is a $\Stab{\MCG(S)}{[\Sigma]}$-equivariant quasi-isometry $\phi_{\Sigma}\colon \C{\Sigma}\to \C{U_\Sigma}$ with uniform constants, mapping each curve $w\in \link{\C S}{\ov \Sigma}$ to its subsurface projection $\rho^{A_w}_{U_\Sigma}$. 
\end{itemize}
To make notation compatible between the two bullets, for $v\in \C S^{(0)}$ we also denote $A_v$ by $U_{\Delta_v}$.

Though we state them for mapping class groups for clarity of exposition, both \Cref{prop:different_proj_in_converse_are_equal} and \Cref{cor:consistency_is_the_same} below hold in general for any combinatorial HHS structure obtained from \cite[Theorem 3.15]{converse}, as their proofs shall use no specific properties of mapping class groups. For the ease of notation, given a simplex $\Sigma$ of $X$ and two subsets $A,B$ of either $ \C{\Sigma}$ or $\C{U_\Sigma}$, we shall write $A\sim B$ to denote that $\diam(A\cup B)$ is bounded independently on $\Sigma$.
    \begin{prop}\label{prop:different_proj_in_converse_are_equal} Let $\Sigma$ be a simplex not of bounded type, and let $v\in (\C S-\Sat(\Sigma))^{(0)}$. Then
    $$p(v)\sim  \phi_\Sigma^{-1}\left(\rho^{A_v}_{U_\Sigma}\right),$$
    where $p\colon Y_\Sigma\to \C{\Sigma}$ is the coarse closest point projection in $(X,\W)$, while $\rho^{A_v}_{U_\Sigma}$ is the subsurface projection in $\frakS$.
    \end{prop}

    \begin{proof} We split the proof in two cases, depending on whether $\Sigma$ is of blowup type or of cusp type.
    
\par\medskip
    \textbf{Blowup type.} Suppose first that $\Sigma$ is of blowup type. Let $\ov \Pi\subset \C S$ be a (possibly empty) simplex whose link contains both $v$ and $\link{\ov X}{\ov \Sigma}$, let $\Pi$ be a simplex of blowup type supported on $\ov \Pi$, and let $\gamma=\{v=w_0,w_1,\ldots,w_n=w\}$ be a $\C\Pi$-geodesic connecting $v$ to a closest point $w\in \phi_\Sigma^{-1}(\rho^{A_v}_{U_\Sigma})\subseteq \link{X}{\Sigma}$. Since $\Pi$ and $\Sigma$ are both of blowup type, we can assume that the geodesic lies entirely in $\C S$. We proceed by induction on $n=\text{dim}(\C S)-|\ov\Pi|$. 
    \par\medskip 
    \emph{Base case $n=1$.} We first suppose that $\ov\Pi$ is almost-maximal. Therefore $\gamma$ does not intersect the saturation of $\Sigma$, because $\link{\C S}{\ov \Pi}$ is discrete while every point in $\Sat(\Sigma)\cap \link{\C S}{\ov \Pi}$ would be $\C S$-adjacent to $w$. Then the bounded geodesic image for the combinatorial HHS structure (see e.g. \cite[Lemma 1.26]{short_HHG:I}) shows that $p(v)\sim w$. We now bound the projection of $v$ and $w$ to $\C{U_\Sigma}$.
    
    \begin{claim}\label{claim:projections_on_surface_are_close}
        $\rho^{A_v}_{U_\Sigma}\sim \rho^{A_w}_{U_\Sigma}$.
    \end{claim}

    \begin{claimproof}[Proof of \Cref{claim:projections_on_surface_are_close}] The path $\gamma'\coloneq \phi_{\Pi}(\gamma)= \{\rho^{A_{v}}_{U_\Pi},\rho^{A_{w_1}}_{U_\Pi}\ldots,\rho^{A_{w}}_{U_\Pi}\}$ is a quasi-geodesic in $\C{U_\Pi}$ with uniform constants. Let $\eta$ be a geodesic in $\C{U_\Pi}$ with the same endpoints, which by the Morse lemma (see, e.g., \cite[III.H.1.7]{BH}) must lie in a uniform neighbourhood of $\gamma'$. The bounded geodesic image axiom in $\frakS$ implies that one of the following is true:
    \begin{itemize}
        \item $\rho^{A_v}_{U_\Sigma}\sim \rho^{A_w}_{U_\Sigma}$, and we are done;
        \item or $\eta$ intersects a uniform neighbourhood of $\rho^{U_\Sigma}_{U_\Pi}$, and therefore so does $\gamma'$. Let $\{\rho^{A_{w_i}}_{U_\Pi},\ldots,\rho^{A_{w_j}}_{U_\Pi}\}$ be the subsegment of $ \gamma'$ contained in this uniform neighbourhood. Since $\gamma'$ is a quasigeodesic with uniform constants, the difference $j-i$ is uniformly bounded. Furthermore, for every $k=i,\ldots, j$, the projection $\rho^{A_{w_k}}_{U_\Sigma}$ is well-defined, since the curve $w_i$ does not belong to the saturation of $\Sigma$ and must therefore intersect $U_\Sigma$. By how $\W$-edges for the combinatorial structure are defined in \cite[Definition 4.10]{converse}, we have that $\rho^{A_{w_k}}_{U_\Sigma}\sim \rho^{A_{w_{k+1}}}_{U_\Sigma}$ for every $k=i,\ldots, j-1$. Finally, again the bounded geodesic image theorem yields that $\rho^{A_{v}}_{U_\Sigma}\sim \rho^{A_{w_{i}}}_{U_\Sigma}$, because the segment of $\gamma'$ between these two points does not intersect the uniform neighbourhood of $\rho^{U_\Sigma}_{U_\Pi}$; the same argument shows that $\rho^{A_{w_j}}_{U_\Sigma}\sim \rho^{A_{w}}_{U_\Sigma}$. Summing up, we again get that $\rho^{A_v}_{U_\Sigma}\sim \rho^{A_w}_{U_\Sigma}$, as required. See \Cref{fig:bgi_in_parallel} to understand the situation. \qedhere
    \end{itemize} 
    \end{claimproof}
    
    \begin{figure}[htp]
        \centering
        \includegraphics[width=\linewidth]{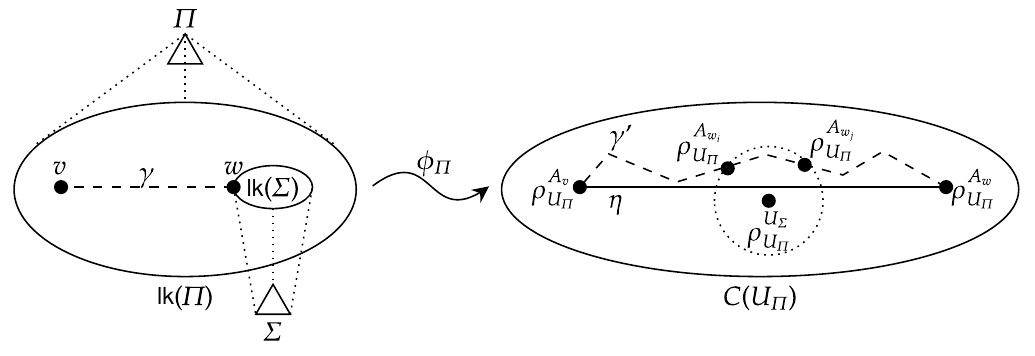}
        \caption{On the left, the situation in $\link{X}{\Pi}$. The geodesic $\gamma$ (the dashed line) only consists of $\W$-edges, because it lies in $\duaug{\link{\C S}{\ov \Pi}}{\W}$ and $\link{\C S}{\ov \Pi}$ is discrete. Furthermore, no point $w_i$ on $\gamma$ belongs to the saturation of $\Sigma$. 
        \\
        On the right, the image of $\gamma$ under the quasi-isometry $\phi_\Pi$ is a quasi-geodesic $\gamma'$, which fellow-travels a geodesic $\eta$ (here, the continuous line). The bounded geodesic image theorem implies that all segments of $\gamma'$ outside a uniform neighbourhood of $\rho^{U_\Sigma}_{U_\Pi}$ (here, the dotted circle) have uniformly bounded projections to $\C{U_\Sigma}$. Furthermore, a uniformly bounded segment of $\gamma'$ can intersect said neighbourhood, and any two points in this segment have close projections to $\C{U_\Sigma}$ by how $\W$-edges are defined in the combinatorial structure.}
        \label{fig:bgi_in_parallel}
    \end{figure}

    \noindent Finally, notice that $w\sim \phi_{\Sigma}^{-1}\left(\rho^{A_w}_{U_\Sigma}\right)$ since the quasi-isometry $\phi_{\Sigma}$ maps $w$ to $\rho^{A_w}_{U_\Sigma}$. Summing everything up, we proved that
    $$p(v)\sim w\sim \phi_{\Sigma}^{-1}\left(\rho^{A_w}_{U_\Sigma}\right)\sim \phi_{\Sigma}^{-1}\left(\rho^{A_v}_{U_\Sigma}\right).$$
    \emph{Inductive step.} Suppose now that $\ov\Pi$ is not almost maximal. If $\gamma$ does not intersect the saturation of $\Sigma$ we conclude exactly as above; otherwise let $\gamma\cap\Sat(\Sigma)$ contain some curve $x$, which must be unique because $\gamma$ is a geodesic and $w\in \link{X}{\Sigma}\subseteq \link{X}{x}$. Let $v'$ be the last vertex of $\gamma$ before $x$. Since the subsegment of $\gamma$ between $v$ and $v'$ does not intersect the saturation of $\Sigma$, the arguments in the base case imply that $p(v)\sim p(v')$ and $\rho^{A_v}_{U_\Sigma}\sim \rho^{A_{v'}}_{U_\Sigma}$. 

    We now produce a vertex $v''\in \link{\C S}{\ov \Pi\star \{x\}}$ such that $p(v')\sim p(v'')$ and $\rho^{A_{v'}}_{U_\Sigma}\sim \rho^{A_{v''}}_{U_\Sigma}$. To this extent, notice that the edge of $\gamma$ connecting $v'$ and $x$ must either comes from $\C S$ or is a $\W$-edge. In the first case $v'\in \link{\C S}{\ov \Pi\star \{x\}}$, so we set $v''=v'$ and we are done. In the second case, there are two $\W$-adjacent maximal simplices $\Psi, \Psi'$ of $X$ such that $x\in \Psi$ and $v'\in \Psi'$. By \Cref{CHHS_def_edge_in_link} we can assume that $\Psi$ and $\Psi'$ both extend $\Pi$. Notice that $\ov\Psi\cap \link{\C S}{\ov \Pi}$ is not entirely contained inside $\Sat(\Sigma)$, since otherwise we would have that $\emptyset \neq\link{\C S}{\ov \Sigma}\subseteq \link{\link{\C S}{\ov \Pi}}{\ov\Psi\cap \link{\C S}{\ov \Pi}}$, and the latter is empty because $\ov\Psi$ is maximal. Hence there exists $v''\in (\ov\Psi\cap \link{\C S}{\ov \Pi})-\Sat(\Sigma)$. Since $v'$ and $v''$ are adjacent in $\C{\Pi}$ and none of them belongs to $\Sat(\Sigma)$, the above arguments, applied to the geodesic $\{v',v''\}$, show that $p(v')\sim p(v'')$ and $\rho^{A_{v'}}_{U_\Sigma}\sim \rho^{A_{v''}}_{U_\Sigma}$.

    Finally, since by construction $v''$ and $\link{X}{\Sigma}$ are both contained in the link of any blowup-type simplex supported on $\ov\Pi\star\{x\}$, the inductive hypothesis gives that $p(v'')\sim \phi_\Sigma^{-1}\left(\rho^{A_{v''}}_{U_\Sigma}\right)$. Summing everything up, we proved that
    \begin{align*}
    p(v)\sim p(v')\sim p(v'')\sim \phi_\Sigma^{-1}\left(\rho^{A_{v''}}_{U_\Sigma}\right) \sim \phi_\Sigma^{-1}\left(\rho^{A_{v'}}_{U_\Sigma}\right)\sim \phi_{\Sigma}^{-1}\left(\rho^{A_v}_{U_\Sigma}\right).\end{align*}
    \textbf{Cusp type.} The case where $\Sigma=\Delta_u$ for some $u\in \C S^{(0)}$ is dealt with similarly, so we focus on highlighting the differences. As above, let $\Pi$ be a simplex of blowup type whose link contains both $v$ and $\link{X}{\Delta_u}$, and let $\gamma$ be a $\C\Pi$-geodesic connecting $v$ to a point $y\in \phi_{\Delta_u}^{-1}(\rho^{A_v}_{A_u})$. We can assume that every vertex of $\gamma$ belongs to $\C S$, excluding $y$ which lies in $\link{X}{\Delta_u}$. Suppose first that $\gamma\cap\Sat(\Delta_u)=\emptyset$, which means that $\gamma$ does not intersect $\link{\C S}{u}$. Let $v'$ be the last vertex of $\gamma$ before $y$. The bounded geodesic image theorem in $(X,\W)$ yields that $p(v)\sim p(v')$. Furthermore, since $v'$ and $y$ are joined by a $\W$-edge, we have that $p(v')\sim y$. Hence
    $$p(v)\sim p(v')\sim y\sim \phi_{\Delta_u}^{-1}(\rho^{A_v}_{A_u}).$$

    Thus suppose that  $\gamma\cap\link{\C S}{u}=\{x\}$. Let $v'$ be the last point on $\gamma$ before $x$, and let $v''\in \link{\C S}{x}$ which is $\W$-adjacent to $v'$ and does not belong to $\link{\C S}{u}$ (again, $v''$ can be found in any maximal simplex containing $x$ which is $\W$-adjacent to a simplex containing $v'$). Applying twice the bounded geodesic image theorem in $(X,\W)$ yields that $p(v)\sim p(v')\sim p(v'')$; on the other hand, the bounded geodesic image theorem in $\frakS$ implies that $\phi_{\Delta_u}^{-1}(\rho^{A_v}_{A_u})\sim \phi_{\Delta_u}^{-1}(\rho^{A_{v'}}_{A_u})\sim \phi_{\Delta_u}^{-1}(\rho^{A_{v''}}_{A_u})$. We then repeat the whole argument for $v''$ inside the link of any blowup type simplex supported on $\ov \Pi\star\{x\}$, and we conclude by induction.
    \end{proof}

    \noindent The following corollary of \Cref{prop:different_proj_in_converse_are_equal} roughly states that relative projections between domains in the combinatorial structure, as defined in \Cref{defn:projections}, coarsely agree with relative projections in the original structure $\frakS$, which are defined using subsurface projections. 
    \begin{cor}\label{cor:consistency_is_the_same} Let $\Sigma,\Theta\subseteq X$ be simplices not of bounded type.
    \begin{enumerate}
        \item\label{item:consistency1} If $[\Theta]\transverse[\Sigma]$ or $[\Theta]\propnest [\Sigma]$, then 
        $$\rho^{[\Theta]}_{[\Sigma]}\sim  \phi_\Sigma^{-1}\left(\rho^{U_\Theta}_{U_\Sigma}\right),$$
        where $\rho^{[\Theta]}_{[\Sigma]}$ is the projection in $(X,\W)$ while $\rho^{U_\Theta}_{U_\Sigma}$ is the subsurface projection in $\frakS$.
        \item\label{item:consistency2}  If $[\Theta]\sqsupsetneq[\Sigma]$ and $x\in \C{\Theta}$, then  
        $$\rho^{[\Theta]}_{[\Sigma]}(x)\sim  \phi_\Sigma^{-1}\left(\rho^{U_\Theta}_{U_\Sigma}(\phi_\Theta(x))\right),$$
        where $\rho^{[\Theta]}_{[\Sigma]}\colon \C{\Theta}\to \C{\Sigma}$ is the projection in $(X,\W)$ while $\rho^{U_\Theta}_{U_\Sigma}\colon \C{U_\Theta}\to \C{U_\Sigma}$ is the subsurface projection in $\frakS$.
    \end{enumerate}
    \end{cor}

    \begin{proof}
    \eqref{item:consistency1} Recall that $\rho^{[\Theta]}_{[\Sigma]}=p(\Sat(\Theta)\cap Y_\Sigma)$. If $\Theta=\Delta_v$ for some $v\in \C S^{(0)}$, then $v\in \Sat(\Theta)\cap Y_\Sigma$, and since $\rho^{[\Theta]}_{[\Sigma]}$ is uniformly bounded it must uniformly coarsely coincide with $p(v)$; then $ p(v)\sim \phi_\Sigma^{-1}\left(\rho^{A_v}_{U_\Sigma}\right)$ by \Cref{prop:different_proj_in_converse_are_equal}, and we are done since $A_v=U_{\Delta_v}$. 
    
    If instead $\Theta$ is of blowup type, let $v\in \ov\Theta$ be such that the projection $\rho^{A_v}_{U_\Sigma}$ is well-defined: such a vertex exists since every subsurface in $\frakS$ which is disjoint from $A_w$ for all $w\in \ov\Theta$ must be nested in $U_\Theta$. Then, since $w\in \Sat(\Theta)\cap Y_\Sigma$, we again have that $\rho^{[\Theta]}_{[\Sigma]}\sim p(v)\sim \phi_\Sigma^{-1}\left(\rho^{A_v}_{U_\Sigma}\right)$. On the other hand, since $A_v$ and $U_\Theta$ are not orthogonal, $\rho^{A_v}_{U_\Sigma}\sim \rho^{U_\Theta}_{U_\Sigma}$ by e.g. \cite[Lemma 1.5]{DHS}. 
\par\medskip
    \eqref{item:consistency2} If one of $\rho^{U_\Theta}_{U_\Sigma}(\phi_\Theta(x))$ and $\rho^{[\Theta]}_{[\Sigma]}(x)$ is empty, then then the diameter of their union is the maximum of the diameters, which is uniformly bounded, and there is nothing to prove. We therefore assume that both subsets are non-empty. Recall that $\rho^{[\Theta]}_{[\Sigma]}(x)$ is non-empty precisely when $x\in \C{\Theta}\cap Y_\Sigma$, and in this case $\rho^{[\Theta]}_{[\Sigma]}(x)=p(x)$. Furthermore, let $v\in \C S^{(0)}$ be such that $x\in \Cone{v}$. Then $\rho^{U_\Theta}_{U_\Sigma}(\phi_\Theta(x))$ is non-empty precisely when $A_v$ intersects $U_\Sigma$, or equivalently if $v\in Y_\Sigma$, and in this case $\rho^{U_\Theta}_{U_\Sigma}(\phi_\Theta(x))\sim \rho^{A_v}_{U_\Sigma}$. Finally, $p(x)\sim p(v)$, as $v$ and $x$ are adjacent in $Y_\Sigma$, and in turn $p(v)\sim \phi_\Sigma^{-1}\left(\rho^{A_v}_{U_\Sigma}\right)$ by \Cref{prop:different_proj_in_converse_are_equal}. Combining the above yields that 
    \begin{align*}\rho^{[\Theta]}_{[\Sigma]}(x)=p(x)\sim p(v) \sim \phi_\Sigma^{-1}\left(\rho^{A_v}_{U_\Sigma}\right)\sim \phi_\Sigma^{-1}\left(\rho^{U_\Theta}_{U_\Sigma}(\phi_\Theta(x))\right). &\qedhere
    \end{align*}
    \end{proof}

\noindent We are now ready to prove that the cusped space for a mapping class group coincides with its Teichm\"{u}ller space. The core idea is that, since the two spaces admit HHS structures with the same domains and coordinate spaces, and since relative projections between domains in the two structure coarsely coincide by \Cref{cor:consistency_is_the_same}, the two spaces must essentially coincide, and indeed the Realization theorem \cite[Theorem 3.1]{HHS_II} produces an explicit quasi-isometry.
\begin{prop} Let $S$ be a finite-type surface, and let $\MCG(S)$ be its mapping class group.
        The cusped space for $\MCG(S)$ is $\MCG(S)$-equivariantly quasi-isometric to the corresponding Teichm\"{u}ller space, equipped with the Teichm\"{u}ller metric.
    \end{prop}

\begin{proof}
    By \cite[Theorem G]{HHS_I}, the HHS structure for $\Teich{S}$ with the Teichm\"{u}ller metric has the same domains and coordinate spaces as that for $\MCG(S)$, except that each annular curve graph is replaced by the combinatorial horoball over it. Relative projections are again defined using subsurface projections, so they are the same as for $\MCG(S)$.
    
    Now let $(X,\W)$ be the combinatorial HHG structure for $\MCG(S)$, and let $(\hat X, \hat \W)$ be the associated cusped space. All projections between domains in $\hat \W$ are defined using the coarse closest point projections $Y_{\Sigma}\to\hatC{\Sigma}$, which away from $\hatC{\Sigma}$ restrict to the projections $Y_{\down{\Sigma}}\to\C{\down{\Sigma}}$, as explained in the proof of \Cref{prop:W_is_relHHS}. In particular, relative projections between domains in $(\hat X, \hat \W)$ coincide with the corresponding relative projections in $(X, \W)$ when restricted to $X$.

    We now define an equivariant quasi-isometry $\hat \W\to \Teich{S}$. For every $w\in \hat W$, let $b_w\coloneq (\pi_{[\Sigma]}(w))\in \prod \hatC{\Sigma}$ be the tuple of its coordinates, where $[\Sigma]$ varies among all simplices not of bounded type. The tuple $b_w$ is \emph{consistent}, which roughly means that it satisfies certain inequalities relating coordinates in $b_w$ and relative projections between domains (see \cite[Definition 1.17]{HHS_II}). Now let $\phi(b_w)\coloneq (\hat \phi_\Sigma(\pi_{[\Sigma]}(w)))\in \prod \hatC{U_\Sigma}$, where $\hatC{U_\Sigma}$ is the coordinate space for $U_\Sigma$ in the structure for $\Teich{S}$ and $\hat\phi_\Sigma\colon\hatC{\Sigma}\to \hatC{U_\Sigma}$ is the quasi-isometry induced by $\phi_\Sigma\colon \C{\Sigma}\to \C{U_\Sigma}$. Notice that $\hat\phi_\Sigma$ is $\Stab{\MCG(S)}{[\Sigma]}$-equivariant since $\phi_\Sigma$ is; hence the whole map $\phi$ is equivariant with respect to the action of $\MCG(S)$. Now \Cref{cor:consistency_is_the_same} implies that $\phi(b_w)$ is again consistent, because relative projections in the combinatorial structure and in the original structure coarsely coincide. In turn, the Realization Theorem for consistent tuples \cite[Theorem 3.1]{HHS_II} produces a point $f(w)\in \Teich{S}$ whose coordinates roughly coincide with $\phi(b_w)$; furthermore $f(w)$ is coarsely unique, meaning that every point with the same properties is within uniformly bounded Teichm\"{u}ller distance from $f(w)$. 
    
    We just defined a $\MCG(S)$-equivariant coarse map $f\colon \hat \W\to \Teich{S}$, and we claim that is coarsely Lipschitz. To see this, let $w,w'\in \hat W$ span an edge. Since coordinate projections are coarsely Lipschitz, and since each $\hat\phi_\Sigma$ is a quasi-isometry, we have that $\hat \phi_\Sigma(\pi_{[\Sigma]}(w))\sim \hat \phi_\Sigma(\pi_{[\Sigma]}(w'))$ for all $[\Sigma]$. Then the uniqueness axiom \cite[Definition 1.1.(9)]{HHS_II} implies that $f(w)\sim f(w')$, as their coordinates are uniformly close.

    The same arguments yield a coarsely Lipschitz, coarse map $g\colon \Teich{S}\to \hat W$, defined by mapping each point $x\in \Teich{S}$ to a realisation point for the tuple $\phi_\Sigma^{-1}(\pi_{U_\Sigma}(x))\in \prod \hatC{\Sigma}$. It is then clear that $f$ and $g$ are quasi-inverses, thus proving that $f$ is a quasi-isometry.
\end{proof}

\begin{rem}[Comparison to the augmented marking complex]\label{rem:augmented_markings}
    In the case of the mapping class group of a surface $S$, our cusped space is nearly identical to the augmented marking complex $AM(S)$, which is a simplicial complex quasi-isometric to the Teichm\"{u}ller space of $S$ constructed by Durham in \cite[Section 3.2]{Durham_augmented_markings}. A vertex $a$ of $AM(S)$ is a \emph{complete clean marking}, as in \cite{Masur_Minsky_2}, together with a choice of a non-negative integer for every curve of the base, which roughly records how short the curve is in the hyperbolic metric associated to $a$. This data roughly corresponds to a maximal simplex $\Sigma(a)\in \hat \W$, whose underlying simplex $\down{\Sigma(a)}$ is the marking and whose depths are given by the length data. Furthermore, edges of $AM(S)$ come in two types: those coming from flip and twist moves of the underlying markings, analogous to the elementary moves in the marking complex for the mapping class group; and those corresponding to increasing the length data of a curve by one, thus moving deeper in the associated horoball. This reflects the same dichotomy from \Cref{defn:cusped_space} between edges of $\W$-type and of cusp-type. 

    Therefore, to show that our cusped space is quasi-isometric to the Teichm\"{u}ller space, we could have made the above analogy precise by proving that the coarse map $\Sigma(\cdot)\colon AM(S)\to \hat \W$ is a quasi-isometry. However, we preferred to argue as above, both to avoid the technicalities arising for the fact that vertices of $\hat \W$ also include non-clean markings, and in order to develop the far more general tool of \Cref{prop:different_proj_in_converse_are_equal}.
\end{rem}

\section{Controlling torsion}\label{sec:control_torsion}
The main result of this Section is \Cref{thm:torsion-of_quot}, which roughly states that, if a group $G$ has a quotient $\widetilde G$ which is a relative \squid{} HHG, then every finite subgroup $\widetilde F\le \widetilde G$ either comes from a finite subgroup of $G$ or permutes a collection of orthogonal minimal domains. Crucially, this dichotomy will follow from the fact that the $\widetilde F$-action on the cusped space for $\widetilde G$ has a uniformly bounded orbit, which either lies in a coarsely embedded copy of $\widetilde G$ or inside a product of horoballs. We first recall a definition.
\begin{defn}
    Let $G$ be a group and $S$ a generating set. The \emph{injectivity radius} of a quotient $\pi\colon G\twoheadrightarrow \widetilde{G}$ with respect to $S$, denoted $\inj{\pi}{S}$, is the maximum radius of a ball in $\Cay{G}{S}$ around the identity which injects inside $\Cay{\widetilde{G}}{\pi(S)}$. 
\end{defn}

\begin{thm}\label{thm:torsion-of_quot} For every $n,\delta,K\ge 0$ there exists $\rho=\rho(n,\delta,K)\ge 0$ such that the following holds.
Let $G$ be a group, $S$ a finite generating set for $G$, $\pi\colon G\twoheadrightarrow \widetilde{G}$ a quotient, and $\widetilde{F}\le \widetilde{G}$ a finite subgroup. Suppose that:
\begin{enumerate}
    \item $\widetilde{G}$ has a relative \squid{} HHG structure $(X,\W,n,\delta)$;
    \item Some orbit map $\Cay{\widetilde{G}}{\pi(S)}\to \W$ is a $K$-quasi-isometry;
    \item $\inj{\pi}{S}\ge \rho$.
\end{enumerate}
Then either $\widetilde F$ fixes a non-empty simplex of $\ov X$ setwise, or is the isomorphic image under $\pi$ of a finite subgroup $F\le G$ of order at most $|S|^{\rho+1}$.
\end{thm}

\noindent Setting $G=\widetilde{G}$ in the above Theorem yields:
\begin{cor}\label{cor:Nielsen_realisation} For every $n,\delta,K\ge 0$ there exists $N\ge 0$ such that the following happens.
Let $G$ have a relative \squid{} HHG structure $(X,\W,n,\delta)$. Suppose that, for some finite generating set $S$ of $G$, some orbit map $\Cay{G}{S}\to\W$ is a $K$-quasi-isometry. Then every finite subgroup of $G$ of order greater than $|S|^N$ must fix a non-empty simplex of $\ov X$ setwise. 
\end{cor}

\begin{proof}[Proof of \Cref{thm:torsion-of_quot}] 
We shall progressively define constants $E_0,\ldots, E_4$, all depending on $n$, $\delta$, and $K$ only, and eventually set $\rho=2E_4$.

Let $(\hat X,\hat \W)$ be the cusped space for $(X,\W)$, which is a \squid{} HHS with constants $n,\delta'=\delta'(\delta)$ by \Cref{prop:teich_is_comb}. By \cite[Theorem A]{HHP:coarse}, there exists $E_0\ge 0$ such that $\hat \W$ is $E_0$-\emph{coarsely injective}. We will not need the full definition of this property; it will suffice to know that, by combining \cite[Proposition 1.1]{HHP:coarse} and \cite[Theorem 1.2]{lang}, every finite group acting by isometries on an $E_0$-coarsely injective space has an orbit of diameter at most $E_1$, where $E_1$ only depends on $E_0$. We stress that the construction in \cite{HHP:coarse} ultimately depends only on the constants and functions appearing in the definition of a hierarchically hyperbolic space \cite[Definition 1.1]{HHS_II}; in turn, for the \squid{} HHS $(\hat X,\hat \W,n,\delta')$, the latter are defined in terms of $n$ and $\delta'$, as we pointed out in \Cref{rem:uniform_CHHS}. This proves that $E_0$, and in turn $E_1$, only depend on $n$ and $\delta$.

Now let $\widetilde{F}\le \widetilde{G}$ be a finite subgroup, which has an orbit of diameter at most $E_1$ in $\hat \W$ by coarse injectivity. The idea we shall expand upon below is that either this orbit is far from $\W$, and therefore has ``deep" projections to some cusps supported on a simplex of $\ov X$, or $\widetilde{F}$ has a uniformly bounded orbit in $\W$, hence in $G$, and therefore its cardinality is uniformly bounded. 

Let $\Theta=(v,p^{m_v})_{v\in \ov \Theta}\in \hat \W^{(0)}$ be a maximal simplex such that $\diam_{\hat \W}(\widetilde{F}\cdot \Theta)\le E_1$. Since projections to coordinate spaces in $\hat \W$ are uniformly coarsely Lipschitz, there exists $E_2$, depending on $n$, $\delta$, and $E_1$, such that for every simplex $\Sigma$ of $\hat X$, 
\begin{equation}\label{eq:def_of_e2}
    \diam_{\hatC{\Sigma}}(\widetilde{F}\cdot \Theta)\le E_2. 
\end{equation}
Now let 
$$\mathcal V=\widetilde{F}\cdot \{v\in \ov\Theta^{(0)}\mid m_v\ge E_2+1\},$$
that is, the orbit of the supports where the depth of $\Theta$ is greater than $E_2$. We claim that any two $u,v\in \mathcal V$ are adjacent in $\ov X$. Indeed, up to the action of $\widetilde{F}$, assume that $v\in \ov\Theta^{(0)}$ and $u\in \widetilde{f}\cdot \ov\Theta^{(0)}$ for some $\widetilde{f}\in \widetilde{F}$. If $u$ and $v$ were not adjacent, then the projection of $\widetilde{f}\cdot \Theta$ to $\hatC{\Delta_v}$ would be contained in $\C{\Delta_v}$, and therefore $$\dist_{\hatC{\Delta_v}}(\widetilde{f}\cdot \Theta,\Theta)= \dist_{\hatC{\Delta_v}}(\C{\Delta_v},p^{m_v})=m_v\ge E_2+1,$$
contradicting the definition of $E_2$ in \Cref{eq:def_of_e2}.

The above shows that $\mathcal V$ spans a simplex  of $\ov X$, which is preserved by $\widetilde{F}$. If $\mathcal V\neq \emptyset$ we are done, so suppose otherwise. We will prove that $\widetilde{F}$ has uniformly bounded cardinality, by showing that it has a uniformly bounded orbit in $\Cay{\widetilde{G}}{\pi(S)}$. To see this, first notice that the cardinality of $\ov\Theta$ is at most the dimension of $\ov X$, which is bounded by $n$ by \Cref{rem:dimension_of_ovX}. Hence, by how we defined $\hat \W$-edges in \Cref{defn:cusped_space}, we see that
$$\dist_{\hat \W}(\Theta,\down{\Theta})\le \sum_{v\in \ov\Theta^{(0)}} m_v\le n E_2,$$
where we used that each $m_v$ is at most $E_2$ since $\mathcal V$ is trivial. This proves that the $\widetilde{F}$-orbit of $\down{\Theta}$ has diameter at most $E_1+2nE_2$ in $\hat \W$; in turn, by \Cref{prop_coarseemb} the diameter of the same orbit in $\W$ is at most $E_3\coloneq (E_1+2nE_2)2^{E_1+2nE_2}$. Since some orbit map $\Cay{\widetilde{G}}{\pi(S)}\to \W$ is a $K$-quasi-isometry, there exists a constant $E_4$ depending on $K$ and $E_3$ (and so ultimately on $K$, $n$, and $\delta$) and some $\widetilde{h}\in \widetilde{G}$ such that $$\diam_{\Cay{\widetilde{G}}{\pi(S)}}(\widetilde{F}\cdot \widetilde{h})\le E_4.$$
Since $\widetilde{G}$ acts by isometries on its Cayley graph, this implies that the conjugate $\widetilde{F}'\coloneq \widetilde{h}^{-1}\widetilde{F}\widetilde{h}$ is contained in the ball of radius $E_4$ around the identity of $\widetilde{G}$. 

Now let $\rho=2E_4$, so that the ball $B$ of radius $2E_4$ around the identity of $G$ injects inside $\widetilde{G}$; in particular, $\pi(B)$ contains $\widetilde{F}'$. Let $F'\coloneq \pi^{-1}(\widetilde{F}')\cap B$, so that $\pi$ restricts to a bijection between $F'$ and $\widetilde{F}'$. Notice that $F'$ is contained in the ball $B'$ of radius $E_4$ around the identity, so $|F'|\le|B'|\le |B|\le |S|^{\rho+1}$. Moreover $F'$ is a subgroup: for every $f,g\in F'$, the product $fg$ still belongs to $B$, since
$$\dist_{\Cay{G}{S}}(fg,1)\le \dist_{\Cay{G}{S}}(f,1)+\dist_{\Cay{G}{S}}(g,1)\le 2E_4=\rho.$$
Hence $fg$ is the unique preimage of $\pi(f)\pi(g)\in \widetilde{F}'$ inside $B$, and must therefore belong to $F'$. Summing up, we proved that $\widetilde{F}'$ is the isomorphic image under $\pi$ of a subgroup of $G$ of order at most $|S|^{\rho+1}$; in turn, since $\widetilde{F}'= \widetilde{h}^{-1}\widetilde{F}\widetilde{h}$, for any $h\in G$ such that $\pi(h)=\widetilde{h}$ the subgroup $F\coloneq hF'h^{-1}$ projects bijectively to $\widetilde{F}$ and has order at most $|S|^{\rho+1}$, as required.
\end{proof}

\noindent We record an immediate application of \Cref{thm:torsion-of_quot} to relatively hyperbolic groups, which recovers \cite[Lemma 4.3]{dahmaniguirardel}.

\begin{cor}\label{cor:df_of_relhyp:torsion} Let $G$ be a finitely generated group which is hyperbolic relative to a finite collection $\mathcal P=\{P_1,\ldots,P_k\}$ of finitely generated peripheral subgroups. For every $i=1\ldots, k$ let $H_i\unlhd P_i$ be a normal subgroup, and let $\mathcal N=\langle \langle H_i\rangle\rangle_{i=1,\ldots, k}$. There exist a finite subset $\mathcal F\subset G$ and a constant $M\ge 0$, both depending only on $G$ and $\mathcal P$, such that, if each $H_i$ intersects $\mathcal F$ trivially, then any finite subgroup of $G/\mathcal N$ is either:
\begin{itemize}
    \item the isomorphic image of a subgroup of $G$ of order at most $M$;
    \item or is conjugated inside the image of a peripheral.
\end{itemize}
\end{cor}

\begin{proof}
Let $S$ be a finite generating set for $G$, and let $\delta_0\ge 0$ be a hyperbolicity constant for the graph obtained from $\Cay{G}{S}$ by adding a cone over each coset of the $P_i$. Let $\rho=\rho(\delta_0,n=3,K=1)$ be the constant given by \Cref{thm:torsion-of_quot}. The relative Dehn filling theorem \cite{Osin_DehnFill,GrovesManning} states that there exists a finite subset $\mathcal F\subset G$ such that, if each $H_i$ misses $\mathcal F$, then $G/\mathcal N$ is hyperbolic relative to $\pi(\mathcal P)= \{P_i/H_i\}_{i=1,\ldots, k}$, and the injectivity radius of $\pi\colon G\to G/\mathcal N$ is at least $\rho$. Now, \Cref{lem:relhyp_is_squid} endows $G/\mathcal N$ with a relative \squid{} structure $(X,\W)$, where:
\begin{itemize}
\item $\W=\Cay{G/\mathcal N}{\pi(S)}$;
\item $\ov X$ is the discrete graph whose vertices are cosets of the $P_i/H_i$;
\item $\duaug{X}{\W}$ is the graph obtained from $\Cay{G/\mathcal N}{\pi(S)}$ by adding a cone over each coset of the $P_i/H_i$.
\end{itemize} In particular, $(X,\W, n,\delta)$ is a relative \squid{} HHS, where $n=3$ and $\delta$ is the hyperbolicity constant of $\duaug{X}{\W}$. In turn, up to enlarging $\mathcal F$, we can ensure that the latter is bounded in terms of $\delta_0$. This can be deduced from \cite[Lemma 5.3]{Osin_DehnFill}, which gives a bound on the relative isoperimetric inequality for the quotient coned-off graph.
 Finally, notice that there is an isometry $\Cay{G/\mathcal N}{\pi(S)}\to \W$, which in particular is a $K$-quasi-isometry with $K=1$. Hence \Cref{thm:torsion-of_quot} implies that any finite subgroup of $G/\mathcal N$ of order greater than $M\coloneq |S|^\rho$ must fix a point of $\ov X$, and is therefore conjugated inside a peripheral.
\end{proof}

\section{Automorphisms of Dehn filling quotients of short HHGs}
\label{sec:Aut}
In this Section we leverage \Cref{thm:torsion-of_quot} to control the automorphism groups of certain quotients of \emph{short HHGs}. These groups, introduced by the first author in \cite{short_HHG:I}, are a broad family of \squid{} HHGs including Artin groups of large and hyperbolic type, RAAGs on triangle- and square-free graphs, graph manifold groups, and the mapping class group of a sphere with five punctures.
\subsection{Background on short HHGs}\label{sec:short_are_squid} We first recall the definition and some properties of short HHGs.
    \begin{defn}\label{def:short_HHG}
        A \emph{short HHG} is a group $G$ admitting a combinatorial HHG structure $(X,\W)$ such that:
        \begin{enumerate}
            \item $X$ is the blowup of a triangle- and square-free graph $\ov X$, such that every connected component of $\ov X$ is not a point. 
            \item $G$ acts on $\ov X$ with finitely many orbits of edges;
            \item For every $v\in\ov{X}^{(0)}$, there is an extension $0\to Z_v\to \Stab{G}{v}\xrightarrow[]{\p_v}H_v\to0$, where $H_v$ is a hyperbolic group and $Z_v$ is a cyclic subgroup of $\Stab{G}{v}$ acting trivially on $\link{\ov{X}}{v}$. 
            \item \label{shortdef_normal} For every $v\in \ov{X}^{(0)}$ and $g\in G$, $Z_{gv}=gZ_vg^{-1}$ .
            \item \label{shortdef_action_on_line} For every $v\in \ov{X}^{(0)}$, $Z_v$ acts geometrically on $\C{\Delta_v}$ and with uniformly bounded orbits on $\C{\Delta_w}$ for every $w\in\link{\ov X}{v}$. 
        \end{enumerate}
    \end{defn}
    \noindent $\ov X$ is called the \emph{support graph}. For every $v\in \ov X^{(0)}$, we call $Z_v$ the \emph{cyclic direction} for $v$. Moreover, \Cref{shortdef_action_on_line} implies that $\C{\Delta_v}$ is a quasiline if $Z_v$ is infinite, and is uniformly bounded otherwise. 
    \begin{rem}
        Notice that short HHGs are \squid{} HHGs by construction: $X$ is a blowup of $\ov X$, and the latter has \cleanish{} simply because it does not contain triangles nor squares. 
    \end{rem}

    \begin{rem}\label{rem:infinite_commute}
    If $v,w\in \ov X^{(0)}$ span an edge, and if $Z_v$ is infinite, then $Z_v$ and $Z_w$ commute. Indeed, every $g\in Z_w$ normalises $Z_v$, so it acts on $Z_v$ as either the identity or the inversion swapping a generator of $Z_v$ with its inverse. However the second possibility does not happen, since $Z_v$ acts loxodromically on $\C{\Delta_v}$ while $Z_w$ acts with uniformly bounded orbits, and therefore cannot ``flip" a quasiaxis for $Z_w$.
\end{rem}
    
    \begin{defn}\label{defn:colourable}
    A short HHG $G$, with support graph $\ov X$, is \emph{colourable} if there exists a partition of the vertices of $\ov{X}$ into finitely many colours, such that no two adjacent vertices share the same colour, and the $G$-action on $\ov X$ preserves the partition. This property, which is a special case of the general notion of a colourable HHG (see e.g. \cite[Section 3]{HP_projection-complexes}), is satisfied by most known short HHGs, including all the examples mentioned at the beginning of this section. 
\end{defn}

\subsection{Dehn filling quotients}\label{subsec:prop_of_dehnfill} We now clarify which quotients of short HHGs we are interested in.
\begin{notation}\label{notation:short_G_and_DT}
    Let $G$ be a colourable short HHG with support graph $\ov X$, as in \Cref{def:short_HHG}. Suppose that all cyclic directions $\{Z_v\}_{v\in \ov X^{(0)}}$ are infinite.
    
    For every $N\in \mathbb{N}_{>0}$ let $DT_N=\langle NZ_v\rangle_{v\in \ov X^{(0)}},$ which is a normal subgroup of $G$, and let $G_N\coloneq G/DT_N$ the associated \emph{Dehn Filling quotient}, with quotient map $\pi_N\colon G\to G_N$. Set $\ov X_N\coloneq \ov X/DT_N$, and for every $v\in \ov X^{(0)}$ let $v_N$ be its image in $\ov X_N$. Let $Z_{v_N}\coloneq Z_v/(DT_N\cap Z_v)$, and let $H_{v_N}\coloneq \Stab{G_N}{v_N}/Z_{v_N}$. Given any short HHG structure $(X,\W)$ with support graph $\ov X$, set $X_N\coloneq X/DT_N$ and $\W_N\coloneq \W/DT_N$.

    We say that a property $P$ holds \emph{if $N$ is deep enough} if there exists $N_0$ such that, whenever $N$ is a non-trivial multiple of $N_0$, $G_N$ satisfies $P$. 
\end{notation}

\noindent For the whole Section, we always work under \Cref{notation:short_G_and_DT}. We now collect and expand results from \cite{short_HHG:II}, our main goal being to show that deep enough Dehn filling quotients fit the framework of \Cref{thm:torsion-of_quot}.

\begin{lem}[{\cite[Corollary 3.40]{short_HHG:II}}]\label{lem_injrad_of_shortquot} The injectivity radius of $\pi_N\colon G\to G_N$ (with respect to any finite generating set for $G$) goes to infinity for $N$ deep enough and $N\to \infty$.
\end{lem}

\noindent The following proposition summarises the structure of deep enough Dehn filling quotients of a short HHG, and it is essentially contained in \cite{short_HHG:II}. The last item is the most subtle one; while it describes bounded coordinate spaces in the HHS structure, we will need to keep track of these (and especially of their diameters as a function of $N$) to obtain uniformity results for sequences of deep enough quotients.

\begin{prop}\label{prop:shortquot_is_short}
    There exists a short HHG structure $(X,\W)$ for $G$, with support graph $\ov X$, such that the following hold if $N$ is deep enough:
\begin{enumerate}[label=(\roman*)]
    \item \label{item:thmshort2:i} $G_N$ has a short HHG structure $(X_N,\W_N)$.
    \item \label{item:thmshort2:ii} For every $v\in \ov X^{(0)}$, $Z_{v_N}$ acts trivially on $\link{\ov X_N}{v_N}$, it is normal in $\Stab{G_N}{v_N}$, and the quotient $H_{v_N}\coloneq \Stab{G_N}{v_N}/Z_{v_N}$ is hyperbolic.
    \item \label{item:thmshort2:iii} For every $v\in \ov X^{(0)}$, $\C{\Delta_{v_N}}$ is quasi-isometric to $\C{\Delta_v}/NZ_v$, and the quasi-isometry constant does not depend on $N$.
    \end{enumerate}
\end{prop}
\begin{proof} Let $(X',\W')$ be any short HHG structure for $G$, with support graph $\ov X$. Arguing exactly as in \cite[Section 4.1]{short_HHG:II}, one can find a natural number $N_0$ and a new short HHG structure $(X,\W)$ such that, for every $v\in \ov X^{(0)}$, the subgroup $\langle N_0Z_w\rangle_{w\in\link{\ov X}{v}}$ acts on $\C{\Delta_v}$ with uniformly bounded orbits. This shows that, if $N$ is deep enough (and in particular a multiple of the above $N_0$), then $\langle NZ_w\rangle_{w\in\link{\ov X}{v}}$ acts on $\C{\Delta_v}$ with uniformly bounded orbits, and the bound is independent on $N$.

    Next, \cite[Theorem 4.1]{short_HHG:II} proves that $(X_N,\W_N)$ is a short HHG structure for $G_N$, thus yielding \Cref{item:thmshort2:i}. Furthermore, \Cref{item:thmshort2:ii} is \cite[Lemma 3.42]{short_HHG:II}.
    
    Finally, \cite[Lemma 4.15]{short_HHG:II} identifies $\C{\Delta_{v_N}}$ with $\C{\Delta_v}/(DT_N\cap \Stab{G}{v})$ (which is denoted by $L_{[v]}$ there). In turn, since we modified the structure in such a way that $\langle NZ_w\rangle_{w\in\link{\ov X}{v}}$ acts on $\C{\Delta_v}$ with uniformly bounded orbits, the space $\C{\Delta_v}/(DT_N\cap \Stab{G}{v})$ is uniformly quasi-isometric to $\C{\Delta_v}/NZ_v$, and this proves \Cref{item:thmshort2:iii}. 
\end{proof}

\noindent The quotients $G_N$ are all hyperbolic, and non-elementary unless $G$ was either virtually cyclic or a product:
\begin{lem}\label{lem:nonelm_hyp_short_quot}
    The following holds if $N$ is deep enough. If $G$ is not virtually cyclic and $\duaug{\ov X}{\W}$ is unbounded, then $G_N$ is non-elementary hyperbolic.
\end{lem}

\begin{proof}
This is simply \cite[Corollary 4.22]{short_HHG:II}. We stress that, with the notation from the proof of \Cref{prop:shortquot_is_short}, it does not matter if we consider the augmented graph in the original structure $(X',\W')$ or in the new one $(X,\W)$, since by \cite[Corollary 2.9]{asdim} both spaces are quasi-isometric to the space obtained from a Cayley graph of $G$ after coning off all cosets of $\ov X$-vertex stabilisers. 
\end{proof}

\noindent On the one hand, the hyperbolicity constant of $G_N$ (meaning, of a Cayley graph with a fixed generating set coming from $G$) must grow as $N$ goes to infinity, as \Cref{lem:order_of_dt_in_shortquot} below shows that $G_N$ contains larger and larger cyclic subgroups. However, the \emph{HHG structures} for the quotients are actually uniform, except for the minimal coordinate spaces:
\begin{lem}\label{lem:shortquot_uniform}
    There exist $n,\delta\ge 0$ such that, if $N$ is deep enough, then $(X_N,\W_N, n,\delta)$ is a \emph{relative} \squid{} HHS.
\end{lem}

\begin{proof}
    This is \cite[Remark 4.8]{short_HHG:II}. Roughly, the only coordinate spaces for the quotient whose hyperbolicity constants depend on $N$ are the quotients $\C{\Delta_v}/NZ_v$, which look like circles of diameter growing linearly in $N$.
\end{proof}

\noindent As an immediate consequence, we get:
\begin{cor}\label{cor:uniform_teich_for_quot}
     Let $G$ be as in \Cref{notation:short_G_and_DT}. For all deep enough $N$, the cusped spaces $(\hat X_N,\hat \W_N)$ are \emph{uniform} \squid{} HHSs, in the sense of \Cref{rem:uniform_CHHS}.
\end{cor}

\noindent We can sum the previous results up to get the following control on torsion elements in deep enough Dehn filling quotients:

\begin{cor}\label{cor:hyp_satisfied}
    Let $G$ be as in \Cref{notation:short_G_and_DT}. If $N$ is deep enough, then every finite subgroup $\widetilde F\le G_N$ satisfies one of the following:
    \begin{itemize}
        \item $\widetilde F$ fixes a vertex of $\ov X_N$, or a pair of adjacent edges;
        \item There exists a finite subgroup $F\le G$ such that $\pi_N\colon G\to G_N$ restricts to an isomorphism $F\to \widetilde F$.
    \end{itemize}
\end{cor}

\begin{proof}
    It is enough to check that the quotient $\pi_N\colon G\to G_N$ satisfies the assumptions of \Cref{thm:torsion-of_quot} whenever $N$ is deep enough. Fix a finite generating set $S$ for $G$, and let $(X,\W)$ be the short HHG structure given by \Cref{prop:shortquot_is_short}. Fix $w\in \W^{(0)}$, and let $K\ge 0$ be such that the orbit map $\phi\colon\Cay{G}{S}\to \W$ sending $g$ to $g\cdot w$ is a $K$-quasi-isometry. Since $\phi$ is $G$-equivariant, the orbit map $\Cay{G_N}{\pi_N(S)}\to \W_N$ sending $\pi_N(g)$ to $\pi_N(g)\cdot w_N$ is also a $K$-quasi-isometry (here $w_N$ is the image of $w$ in $\W_N$). Furthermore, by \Cref{lem:shortquot_uniform} there exist constants $\delta$ and $n$, not depending on $N$, such that $(X_N,\W_N,n,\delta)$ is a relative \squid{} HHS whenever $N$ is deep enough. Finally, \Cref{lem_injrad_of_shortquot} states that the injectivity radius of $G\to G_N$ diverges as $N\to \infty$.
\end{proof}

\noindent We can also control torsion elements in $H_{v_N}$, as the following results recognise $H_{v_N}$ as a quotient of $H_v$ satisfying the requirements of \Cref{cor:df_of_relhyp:torsion}:
\begin{lem}[{\cite[Lemma 2.17]{short_HHG:I}}]
    Let $v\in \ov X^{(0)}$, and let $\mathcal U\subseteq\link{\ov X}{v}$ be any collection of $\Stab{G}{v}$-orbit representatives. Then $H_{v}$ is hyperbolic relative to the collection $\{E(Z_u)\}_{u\in \mathcal W}$, where $E(Z_u)$ is the maximal virtually cyclic subgroup containing $q(Z_u)$.
\end{lem}
\begin{lem}[{\cite[Lemma 3.42]{short_HHG:II}}]
\label{lem:H_v}  
The following holds if $N$ is deep enough. Let $v\in \ov X^{(0)}$, and let $q\colon \Stab{G}{v}\to H_{v}$ be the quotient map. Then $H_{v_N}$ is the quotient of $H_v$ by $\langle Nq(Z_u)\rangle_{u\in \link{\ov X}{v}}$.
\end{lem}
\noindent Combining \Cref{lem:H_v} with \Cref{cor:df_of_relhyp:torsion} we get:
\begin{cor}\label{cor:torsion_of_hv}
    The following holds if $N$ is deep enough. For every $v\in \ov X^{(0)}$, every finite subgroup of $H_{v_N}$ either comes from a finite subgroup of $H_v$, or belongs to the image of $E(Z_w)$ for some $w\in \link{\ov X}{v}$, and therefore fixes $w_N$.
\end{cor}

\subsubsection{Cyclic directions for the quotient}
For later purposes, we gather here a few more lemmas about Dehn filling quotients, again working in the setting of \Cref{notation:short_G_and_DT}. The first lemma describes the image of cyclic directions in the quotient. 

\begin{lem}\label{lem:order_of_dt_in_shortquot} The following holds if $N$ is deep enough. For every $v\in \ov X^{(0)}$, $DT_N\cap \Stab{G}{v}=\langle NZ_w\rangle_{w\in \CStar{v}}$ and $DT_N\cap Z_v=NZ_v$. In particular, $Z_{v_N}\cong \Z/N\Z$.
\end{lem}

\begin{proof}
    The first conclusion is \cite[Corollary 3.36]{short_HHG:II}. For the second equality let $g\in DT_N\cap Z_v$. Since $Z_v\le \Stab{G}{v}$, by the first conclusion $g\in \langle NZ_w\rangle_{w\in \CStar{v}}$; moreover $NZ_v$ is normal in $\Stab{G}{v}$, so we can write $g=hg'$ where $h\in NZ_v$ and $g'\in  \langle NZ_w\rangle_{w\in \link{\ov X}{v}}\cap Z_v$. If $g'$ is trivial we are done, so suppose otherwise towards a contradiction. By \cite[Corollary 3.37]{short_HHG:II} we have $$\langle NZ_w\rangle_{w\in \link{\ov X}{v}}=\bigast_{u\in J}NZ_u,$$
    where $J$ is a certain (possibly infinite) subset of $\link{\ov X}{v}$. If $J$ consists of a single vertex $w$ then $g'\in NZ_w\cap Z_v$, and therefore this intersection is non-trivial. This is a contradiction, because the two cyclic directions have different actions on the quasiline $\C{\Delta_v}$ by \Cref{def:short_HHG}.\eqref{shortdef_action_on_line}. Thus suppose that $J$ contains more than a vertex, and therefore $\langle NZ_w\rangle_{w\in \link{\ov X}{v}}$ is a non-abelian free group. Since $Z_v$ is normal in $\Stab{G}{v}$, the square power of $g'$ lies in the centre of $\langle NZ_w\rangle_{w\in \link{\ov X}{v}}$, which is trivial; in turn, since a non-abelian free group has no torsion, $g'$ itself must be trivial, against our initial assumption.
\end{proof}

\subsubsection{Commuting elements}
Next, we prove that commuting elements of $G_N$ cannot have large displacements on transverse lines: this is an instance of a more general principle that holds for all HHG. For the following statement, the \emph{minimum displacement} of an isometry $g$ of a metric space $S$ is defined as $\inf_{x\in S}\dist_S(x,gx)$.
\begin{lem}\label{lem:mindispl_and_commutation}
    There exists $\widetilde{E}> 0$, only depending on $G$, such that the following holds if $N$ is deep enough. For $i=1,2$ let $v_N^i\in \ov X_N^{(0)}$ and $g_i\in \Stab{G_N}{v^i_N}$. Suppose that the minimum displacement of $g_i$ on $\C{\Delta^i}$ is at least $\widetilde{E}$. If $g_1$ and $g_2$ commute then $v^1_N\in \CStar{v^2_N}$.
\end{lem}

\begin{proof}
Towards a contradiction, assume that $v^1_N$ and $v^2_N$ are not adjacent, so that they yield transverse domains $[\Delta_{v^1_N}]$ and $[\Delta_{v^2_N}]$. Let $r^1_2\coloneq \rho^{[\Delta_{v^1_N}]}_{[\Delta_{v^2_N}]}\subset \C{\Delta_{v^2_N}}$ and $r^2_1\coloneq \rho^{[\Delta_{v^2_N}]}_{[\Delta_{v^1_N}]}\subset \C{\Delta_{v^1_N}}$ be the relative projections between these two transverse domains. Since $(X_N,\W_N,n,\delta)$ is a relative \squid{} HHS, there exists a constant $E=E(n,\delta)>0$, not depending on $N$, such that the following hold:
\begin{itemize}
    \item Both $r^1_2$ and $r^2_1$ have diameter at most $E$ (this is \cite[Definition 1.14.(1)]{HHS_II}).
    \item For every $q\in \W_N$, $\min\left\{\dist_{\C{\Delta_{v^1_N}}}(q,r^2_1), \dist_{\C{\Delta_{v^1_N}}}(q,r^1_2)\right\}\le E$ (this is the Behrstock inequality, \cite[Definition 1.14.(4)]{HHS_II}).
    \item There exists $p\in \W_N$ such that $\max\left\{\dist_{\C{\Delta_{v^1_N}}}(p,r^2_1), \dist_{\C{\Delta_{v^1_N}}}(p,r^1_2)\right\}\le E$ (this follows from the partial realisation axiom \cite[Definition 1.14.(8)]{HHS_II}, applied to the pair $([\Delta_{v^1_N}],r^2_1)$).
\end{itemize}
Now set $\widetilde E=4E$, and let $g_1,g_2\in G_N$ be as in the statement. Notice first that, for $p$ as in the third bullet,
$$\dist_{\C{\Delta_{v^1_N}}}(r^2_1,g_1p)\ge \dist_{\C{\Delta_{v^1_N}}}(p,g_1 p)-\dist_{\C{\Delta_{v^1_N}}}(p,r^2_1)-\diam(r^2_1)\ge 4E-E-E=2E,$$
so the Behrstock inequality yields that $\dist_{\C{\Delta_{v^2_N}}}(r^1_2,g_1 p)\le E$. If we now consider the action of $g_2 $ on $\C{\Delta_{v^2_N}}$, the same reasoning yields 
$$\dist_{\C{\Delta_{v^2_N}}}(r^1_2,g_2 g_1 p)\ge \dist_{\C{\Delta_{v^2_N}}}(g_1 p,g_2 g_1 p)-\dist_{\C{\Delta_{v^2_N}}}(g_1 p,r^1_2)-\diam(r^1_2)\ge 2E.$$
However $g_2 g_1 p=g_1 g_2 p$, so if in the above argument we swap the roles of $v^1_N$ and $v^2_N$, and of $g_1 $ and $g_2 $, we get that $$\dist_{\C{\Delta_{v^1_N}}}(r^2_1,g_2 g_1 p)=\dist_{\C{\Delta_{v^1_N}}}(r^2_1,g_1 g_2 p)\ge 2E,$$ contradicting the Behrstock inequality for $q=g_2 g_1 p$.
\end{proof}

\begin{cor}\label{cor:commuting_multitwist}
    The following hold if $N$ is deep enough. For $i=1,2$ let $\{v^i,w^i\}$ be edges of $\ov X$.
    \begin{enumerate}
        \item \label{item_comm_mult1} If there exist commuting elements $g_1\in Z_{v^1_N}$ and $g_2\in Z_{v^2_N}$, then $v^1_N\in \CStar{v^2_N}$.
        \item \label{item_comm_mult2} If there exist commuting elements $g_1\in \langle Z_{v^1_N}, Z_{w^1_N}\rangle-(Z_{v^1_N}\cup Z_{w^1_N})$ and $g_2\in \langle Z_{v^2_N}, Z_{w^2_N}\rangle-(Z_{v^2_N}\cup Z_{w^2_N})$, then $\{v^1_N,w^1_N\}=\{v^2_N,w^2_N\}$.
    \end{enumerate}
\end{cor}

\begin{proof} For every $u\in \ov X^{(0)}$, $Z_u$ acts $R_u$-coboundedly on $\C{\Delta_u}$ for some $R_u\ge 0$, and since there are finitely many $G$-orbits of vertices in $\ov X$ there is a uniform bound $R$ on all $R_u$. 
Furthermore, by \Cref{prop:shortquot_is_short}.\ref{item:thmshort2:iii}, each $\C{\Delta_{u_N}}$ is uniformly quasi-isometric to $\C{\Delta_{u}}/NZ_u$, so $Z_{u_N}$ acts on $\C{\Delta_{u_N}}$ $R'$-coboundedly for some constant $R'$ depending on $R$ but not on $N$. The same \Cref{prop:shortquot_is_short}.\ref{item:thmshort2:iii} implies that the diameter of $\C{\Delta_{u_N}}$ grows linearly in $N$. Since neither $R'$ nor the constant $\widetilde{E}$ from \Cref{lem:mindispl_and_commutation} depend on $N$, we can find a constant $M$, again not depending on $N$, such that if $g\in Z_{u_N}$ has minimum displacement less than $\widetilde{E}$ on $\C{\Delta_{u_N}}$, then $g=\pi_N(z_u^k)$ for some $|k|\le M$, where $z_u$ generates $Z_u$. 

Now, if $N>4M$, every $g\in Z_{u_N}$ has a power outside $\{\pi_N(z_u^k)\}_{k=-M,\ldots,M}$. This shows that, for $N$ deep enough, any two elements $g_1,g_2$ as in \Cref{item_comm_mult1} admit powers with minimum displacement at least $\widetilde{E}$ on the corresponding coordinate spaces, and now we can conclude by \Cref{lem:mindispl_and_commutation}. 

The argument for \Cref{item_comm_mult2} is similar: given $v^i$, $w^i$, and $g_i$ as in the statement, it is enough to show that, for $i=1,2$, $g_i$ has a power acting with minimum displacement at least $\widetilde{E}$ on both $\C{\Delta_{v^i_N}}$ and $\C{\Delta_{w^i_N}}$, provided that $N$ is deep enough that $N>4M^3$. For ease of notation we drop the indices and denote $v^1=v$, $w^1=w$, and $g_1=g$ (the argument for $i=2$ is identical). Suppose that $g=\pi_N(z_v^r z_w^s)$, for some $-N/2<r,s<N/2$ and $r,s\neq 0$. If $g$ already lies outside $\{\pi_N(z_v^k z_w^h)\}_{k,h=-M,\ldots,M}$ then it acts with minimum displacement at least $\widetilde{E}$ on both $\C{\Delta_{v_N}}$ and $\C{\Delta_{w_N}}$, and we conclude as above. Otherwise suppose that $|r|\le M$, and look at $g^M=\pi_N(z_v^{rM} z_w^{sM})$, so that $M<|rM\mod N|\le M^2$. There are now three sub-cases to consider. If $|sM\mod N|>M$ we are done. If $0<|sM\mod N|\le M$ then $g^{M^2}=\pi_N(z_v^{rM^2} z_w^{sM^2})$ lies outside $\{\pi_N(z_v^k z_w^h)\}_{k,h=-M,\ldots,M}$ by our choice of $N$. Finally, if $sM=0\mod N$ then $|s|\ge 4M^2>M$, so $g^{M+1}=\pi_N(z_v^{r(M+1)} z_w^{s})$ lies outside $\{\pi_N(z_v^k z_w^h)\}_{k,h=-M,\ldots,M}$.
\end{proof}

\noindent From \Cref{cor:commuting_multitwist} we can also extract that cyclic directions can be algebraically distinguished from products of adjacent cyclic directions:
\begin{cor}\label{lem:recognise_(multi)twist} The following holds if $N$ is deep enough. Let $v,w$ be $\ov X$-adjacent vertices of valence greater than one, and let $x\in \langle Z_{v_N}, Z_{w_N}\rangle$. Then $x\in Z_{v_N}\cup  Z_{w_N}$ if and only if its centraliser $C_{G_N}(x)$ is infinite.
\end{cor}

\begin{proof} We first notice that, under the above assumption, $H_v$ is non-elementary hyperbolic. To see this, let $u\in \link{\ov X}{v}-\{w\}$, which exists as $v$ has valence greater than one. If $H_v$ was virtually cyclic then $Z_u$ would intersect $\langle Z_v,Z_w\rangle$ non-trivially. This is impossible since $\langle Z_v,Z_w\rangle$ fixes $w$ while no element in $Z_u$ does: indeed, every element in $Z_u$ acts loxodromically on $\C{\Delta_u}$, and therefore cannot fix the projection of $w$ to $\C{\Delta_u}$. 

Moving to the proof of the Lemma, we first suppose without loss of generality that $x\in Z_{v_N}$. Then $C_{G_N}(x)$ contains the index-two subgroup of $\Stab{G_N}{v_N}$ centralising $Z_{v_N}$, so it is enough to prove that $\Stab{G_N}{v_N}$ is infinite. In turn, the latter fits into an extension 
$$0\to Z_{v_N}\to \Stab{G_N}{v_N}\to H_{v_N}\to 0,$$
so we are left to prove that $H_{v_N}$ is infinite. \Cref{lem:H_v} describes $H_{v_N}$ as a relative hyperbolic Dehn filling of the non-elementary hyperbolic group $H_v$; hence, by \cite[Corollary 1.6]{Osin_DehnFill}, if $N$ is deep enough then $H_{v_N}$ is non-elementary hyperbolic, and in particular infinite. 

For the converse implication, let $x\in \langle Z_{v_N}, Z_{w_N}\rangle-(Z_{v_N}\cup Z_{w_N})$, and let $g\in C_{G_N}(x)$. Then $x=gxg^{-1}\in \langle Z_{g(v_N)}, Z_{g(w_N)}\rangle-(Z_{g(v_N)}\cup Z_{g(w_N)})$, by how cyclic directions behave under conjugation; thus \Cref{cor:commuting_multitwist}.\eqref{item_comm_mult2} implies that $g$ fixes $\{v_N,w_N\}$ setwise. This proves that $C_{G_N}(x)$ is virtually contained in $\Stab{G_N}{v_N}\cap\Stab{G_N}{w_N}$, which in turn virtually coincides with the finite subgroup $\langle Z_{v_N}, Z_{w_N}\rangle\cong Z_{v_N}\times Z_{w_N} $ by \cite[Corollary 2.17]{short_HHG:II}.
\end{proof}

\subsection{Extracting combinatorial data from automorphisms}
We now prove that group automorphisms of $G_N$ induce simplicial automorphisms of $\ov X_N$:
\begin{thm}\label{thm:extract}
    Let $G$ be a colourable short HHG as in \Cref{notation:short_G_and_DT}. Suppose further that $\ov X$ has no valence-one vertices. There exists an integer $M$, only depending on $G$, such that the following holds if $N$ is deep enough. Every group automorphism $\phi\in\Aut{G_N}$ induces a simplicial automorphism $\phi_{\#}\in \Aut{\ov X_N}$ such that, for every $v\in \ov X^{(0)}$, $\phi_{\#}(v_N)$ is the unique vertex of $\ov X_N$ satisfying $$\phi(MZ_{v_N})=MZ_{\phi_{\#}(v_N)}.$$
\end{thm}

\begin{proof} Let $\phi$ be an automorphism of $G_N$. For every $v\in \ov X^{(0)}$, $Z_{v_N}\coloneq Z_v/(DT_N\cap Z_v)$ is isomorphic to $\Z/N\Z$ by \Cref{lem:order_of_dt_in_shortquot}, and therefore $\phi(Z_{v_N})\cong \Z/N\Z$ since $\phi$ is an automorphism. Let $g$ generate $\phi(Z_{v_N})$. Our next goal is to invoke \Cref{cor:hyp_satisfied} to show that a uniform power of $g$ fixes an edge of $\ov X_N$; later we will choose $\phi_{\#}(v_N)$ as one of the endpoints of such edge. 

    \begin{claim}\label{claim:extracting_edge} There exists $M\in \N$, only depending on $G$, such that the following holds for deep enough $N$. There exist $\ov X$-adjacent vertices $v',w'$ such that $g^{M}\in \langle Z_{v'_N}, Z_{w'_N}\rangle$.
    \end{claim}

    \begin{claimproof}[Proof of \Cref{claim:extracting_edge}]
    Firstly, recall that a hierarchically hyperbolic group contains finitely many conjugacy classes of finite subgroups, by \cite[Theorem G]{HHP:coarse}. In particular, we can choose $N$ deep enough that $G$ contains no subgroup isomorphic to $\Z/N\Z$. Hence, by \Cref{cor:hyp_satisfied} there exists a non-empty simplex $\ov\Sigma$ of $\ov X$ which is preserved by $\phi(Z_{v_N})$. Notice that $\ov X_N$ is triangle-free, and therefore $\ov \Sigma$ is either a single vertex or an edge; either way, there exists $v'\in \ov X^{(0)}$ such that $g^2\in \Stab{G_N}{v'_N}$. 

Now let $q\colon \Stab{G_N}{v'_N}\to H_{v'_N}\coloneq \Stab{G_N}{v'_N}/Z_{v'_N}$ be the quotient map. By \Cref{cor:torsion_of_hv} there exists a constant $M'$ such that either $q(g^2)$ has order at most $M'$, or $q(g^2)\in E(Z_{w'})/NZ_{w'}$ for some $w'\in \link{\ov X}{v'}$. In the former case $g^{2M'}\in Z_{v'_N}$; in turn, since $M'$ only depends on the isomorphism type of $H_{v'}$, of which there are finitely many since $G$ acts cofinitely on the vertices of $\ov X$, there exists a power of $g$, only depending on $G$, which belongs to $Z_{v'_N}$. In the latter case, since $Z_{w'}$ has finite-index in $E(Z_{w'})$, a uniform power of $q(g^2)$ must belong to $q(Z_{w'})$, and again we can choose this power to not depend on either $v'$ or $w'$ as there are finitely many $G$-orbits of edges in $\ov X$. 
\end{claimproof}
\noindent Now, since $\ov X$ has no valence-one vertices, \Cref{lem:recognise_(multi)twist} yields that the centraliser of any element in $Z_{v_N}$ is infinite, and the same must be true for $\phi(Z_{v_N})$ as $\phi$ is an isomorphism. In particular $g^M$ has infinite centraliser, and again  \Cref{lem:recognise_(multi)twist} implies that $g^{M}\in Z_{\phi_{\#}(v_N)}$ where $\phi_{\#}(v_N)$ is either $v'_N$ or $w'_N$. Notice moreover that, since $M$ does not depend on $N$, we can choose $N$ to be a multiple of $M$. Then $\langle g^M\rangle=\phi(MZ_v)$ is a subgroup of $Z_{\phi_{\#}(v_N)}\cong \Z/N\Z$ of order $N/M$, and it must therefore coincide with $MZ_{\phi_{\#}(v_N)}$. 

We also point out that $\phi_{\#}(v_N)$ is unique. Indeed, suppose towards a contradiction that $g^{M}\in Z_{v''_N}$ for some $v''_N\neq \phi_{\#}(v_N)$, so that $Z_{\phi_{\#}(v_N)}\cap Z_{v''_N}$ have a non-trivial intersection. \Cref{cor:commuting_multitwist}.\eqref{item_comm_mult1} would then imply that $v''_N\in\link{\ov X_N}{ \phi_{\#}(v_N)}$; but then $Z_{v''_N}$ would act geometrically on $\C{\Delta_{v''_N}}$ and with uniformly bounded orbits on $\C{\Delta_{\phi_{\#}(v_N)}}$, contradicting the existence of a non-trivial intersection.

We are left to prove the following:
\begin{claim}\label{claim:graphiso}
    Whenever $N$ is deep enough, $\phi_{\#}\colon \ov X\to \ov X$ is a graph isomorphism.
\end{claim}
\begin{claimproof}[Proof of \Cref{claim:graphiso}] It is enough to prove that $\phi_{\#}$ maps edges to edges, since then the same argument applied to $\phi^{-1}$ will produce an inverse fof $\phi_{\#}$. Hence let $v_N, w_N$ be $\ov X_N$-adjacent. By  \Cref{rem:infinite_commute}, the subgroups $MZ_{v_N}$ and $MZ_{w_N}$ commute, and therefore $MZ_{\phi_\#(v_N)}$ and $MZ_{\phi_\#(w_N)}$ must commute as well. Now \Cref{cor:commuting_multitwist} shows that $\phi_\#(v_N)\in\CStar{\phi_\#(w_N)} $; moreover, since cyclic directions have trivial intersection, we get that $\phi_\#(v_N),\phi_\#(w_N)$ are $\ov X_N$-adjacent, as required.
\end{claimproof}
\noindent The proof of \Cref{thm:extract} is now complete.
\end{proof}

\noindent We now apply the above theorem to the \emph{extended mapping class group} of a five-punctured sphere (where we allow orientation-reversing mapping classes) to prove an analogue of a celebrated theorem of Ivanov \cite{Ivanov:autC}:
\begin{cor}\label{cor:autmcgs5}
    Let $G=\MCG^\pm(S_5)$ be the extended mapping class group of a five-holed sphere. If $N$ is deep enough, then $\Aut{G_N}=G_N$.
\end{cor}

\begin{proof}
    We first check that $G$ fits the framework of \Cref{thm:extract}. Indeed, in \cite[Section 2.3.1]{short_HHG:I} the first author argued that $G$ admits a colourable short HHG structure, where:
    \begin{itemize}
    \item the support graph $\ov X$ is the curve graph, which has no valence-one vertices;
    \item cyclic directions are the infinite subgroups generated by half Dehn twists. 
    \end{itemize}
    Now let $G_N\to \Aut{G_N}$ be the natural map, sending every element to the corresponding inner automorphism, and we want to show that this map is an isomorphism. For surjectivity, let $\phi\in \Aut{G_N}$ be an automorphism, and let $\phi_{\#}\in \Aut{\ov X_N}$ be the associated graph automorphism from \Cref{thm:extract}. By \cite[Theorem 5.11]{rigidity_mcg_mod_dt}, the map $G_N\to \Aut{X_N}$ induced by the action is an isomorphism. In particular there exists $g\in G_N$ such that $\phi_{\#}$ coincides with the action of $g$ on $\ov X_N$. This implies that, for every $v\in \ov X^{(0)}$, $$\phi(MZ_{v_N})=MZ_{g(v_N)}.$$
    Arguing as in \cite[Section 3]{Ivanov:autC}, we will now show that $\phi$ coincides with the conjugation by $g$. Indeed, for every $f\in G$,
    $$\phi(MfZ_{v_N}f^{-1})=\phi(MZ_{f(v_N)})=MZ_{gf(v_N)},$$
    where we used that $fZ_{v_N}f^{-1}=Z_{f(v_N)}$ by how half Dehn twists behave under conjugation. However
    $$\phi(MfZ_{v_N}f^{-1})=\phi(f)\phi(MZ_{v_N})\phi(f)^{-1}=MZ_{\phi(f)g(v_N)}.$$
    Since the cyclic directions of two distinct vertices of $\ov X_N$ intersect trivially, the above computation shows that $gf(v_N)=\phi(f)g(v_N)$ for all $v\in \ov X^{(0)}$, or in other words $\phi(f)=gfg^{-1}$ as automorphisms of $\ov X_N$. Then again \cite[Theorem 5.11]{rigidity_mcg_mod_dt} shows that $\phi(f)$ must coincide with $gfg^{-1}$ for all $f\in G_N$, thus proving that $\phi$ is the conjugation by $g$.

    We are left to prove injectivity of the map $G_N\to \Aut{G_N}$, or equivalently that $G_N$ has no centre. To see this, notice that, if $g\in G_N$ and the image of any half Dehn twist commute, then the above arguments show that $g$ acts trivially on $\ov X_N$, and again \cite[Theorem 5.11]{rigidity_mcg_mod_dt} implies that $g$ must be trivial.
\end{proof}

\begin{rem}\label{rem:out_is_finite} We can rephrase the above Theorem by saying that the outer automorphism group of $G_N$ is trivial. However, the fact that $\text{Out}(G_N)$ is \emph{finite} could already be extracted from results in the literature. Indeed, by \cite[Theorem 5.8.(2)]{dfdt}, $G_N$ is non-elementary hyperbolic, so classical work of Paulin and Bestvina-Feighn, among others \cite{Paulin, Bestvina-Feighn}, proves that $G_N$ has finite outer automorphism group unless it admits a splitting over virtually cyclic subgroups. The latter does not happen because mapping class groups have Serre's property (FA) by \cite{FA_MCG}, and this property passes to quotients. 
\end{rem}

\section{Distinguishing braid group and spheres}
\label{sec:spheres}
We conclude the paper by proving that the four strand braid group (resp. the mapping class group of a sphere with five punctures) can be recognised from the braid groups on more strands (resp. the mapping class groups of spheres with more punctures) by its non-elementary hyperbolic quotients. This will prove the two parts of \Cref{thmintro:distinguish}.

We first set up some notation and recall some well-known facts, referring to \cite{FM_primer} for all the details. 

\textbf{The groups}. For $n\ge 4$, let $S_n$ be a sphere with $k$ punctures, let $\MCG(S_n)$ be its usual mapping class group (so we do not allow orientation-reversing homeomorphisms), and let $\MCG(S_n,*)$ be the finite-index subgroup of $\MCG(S_n)$ fixing a given puncture. Also, let $B_n$ be the braid group on $n$ strands, which is the mapping class group of a $n$-punctured disk $D_n$. The two mapping class groups are related via the \emph{capping homomorphism}, giving a short exact sequence 
\begin{equation}
	0\to \langle T\rangle \to B_n\to \MCG(S_{n+1},*)\to 0
\end{equation}
where $T$ is the Dehn twist around the boundary of $D_n$; in turn, forgetting the puncture gives a surjection $\MCG(S_{n+1},*)\to \MCG(S_n)$, called the \emph{forgetful map}.

\textbf{Generators}.  $B_n$ is generated by the half Dehn twists $\tau=\tau_1,\ldots \tau_{n-1}=\tau'$ on the curves $\gamma_1,\ldots, \gamma_{n-1}$ from the left of \Cref{fig:gen_of_bn}. Furthermore, a (slightly redundant) generating set for the \emph{pure braid group} $PB_n$, which is the finite-index subgroup of $B_n$ fixing all punctures, is given by the Dehn twists around the curves 
$\{\eta^\pm_{i,j}\}_{1\le i<j\le n}$ from the right of \Cref{fig:gen_of_bn} (notice that $\gamma_{i}=\eta^+_{i,i+1}=\eta^-_{i,i+1}$ for all $i=1,\ldots, n-1$). Any two half Dehn twists are conjugate inside $B_n$. 

    \begin{figure}[htp]
        \centering
        \includegraphics[width=\linewidth]{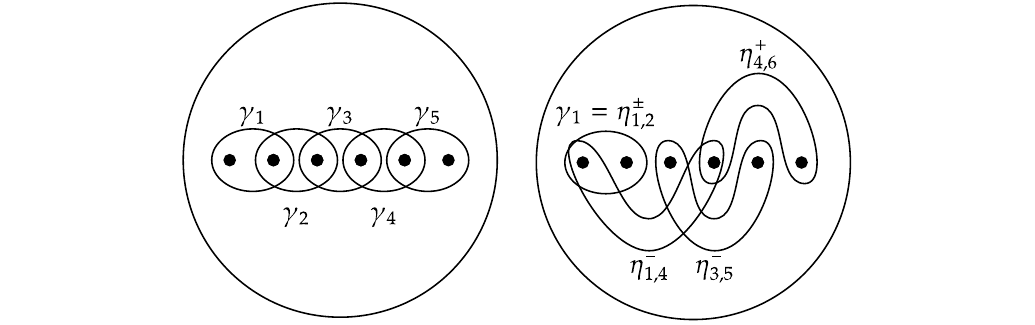}
        \caption{On the left, the half Dehn twists around the curves $\gamma_i$ generate the braid group. On the right, the pure braid group is  generated by the Dehn twists around the curves $\eta^\pm_{i,j}$ for $1\le i<j\le n$ (here we only drew some). $\eta^+_{i,j}$ passes above the punctures between the $i$th and the $j$th, while $\eta^-_{i,j}$ passes  below them.}
        \label{fig:gen_of_bn}
    \end{figure}

\textbf{Abelian quotients}.  The abelianisation of $B_n$ is infinite cyclic, generated by the image of $\tau$, and $T$ maps to $n(n-1)$ inside the abelianisation. As a consequence, the abelianisations of both $\MCG(S_{n+1},*)$ and $\MCG(S_n)$ are finite cyclic.

\textbf{Torsion}.  Finite-order elements in $\MCG(S_5)$ are topologically conjugate to rotations, fixing at most two punctures and cyclically permuting the others, and therefore have order at most $5$ (see e.g. \cite[Section 7.1.1]{FM_primer}). Furthermore, the subgroup fixing a pair of disjoint curves is generated by the half Dehn twists around those two curves.

The following is the main technical statement of the section:
\begin{thm}\label{thm:ubertheorem}
    Let $n\ge 6$ (resp. $n=5$). Then there exists a non-elementary hyperbolic quotient $Q$ of $\MCG(S_5)$ (resp. $\MCG(S_5,*)$) such that any homomorphism $B_n\to Q$ has virtually cyclic image.
\end{thm}
\noindent Before proving \Cref{thm:ubertheorem}, we show that it implies both parts of \Cref{thmintro:distinguish}.

\begin{cor}\label{cor:uber_for_braids}
    For every $n\ge 5$, there exists a non-elementary hyperbolic quotient $Q'$ of $B_4$ such that any homomorphism $B_n\to Q'$ has virtually cyclic image.
\end{cor}

\begin{proof} Let $Q$ be the quotient of $\MCG(S_5)$ from \Cref{thm:ubertheorem} (resp. of $\MCG(S_5,*)$ if $n=5$), and let $Q'$ be the image of $\MCG(S_5,*)$ inside $Q$ (resp. let $Q'=Q$ if $n=5$). Since $Q'$ has finite index in $Q$, it is also non-elementary hyperbolic, and it is a quotient of $\MCG(S_5,*)$, hence of $B_4$ via the capping homomorphism. Furthermore, every map $B_n\to Q'$ is in particular a map $B_n\to Q$, and must therefore have virtually cyclic image.
\end{proof}

\begin{cor}\label{cor:uber_for_spheres}
    For every $n\ge 6$, there exists a non-elementary hyperbolic quotient $Q$ of $\MCG(S_5)$ such that any homomorphism $\MCG(S_n)\to Q$ has finite image.
\end{cor}
\begin{proof}
Let $Q$ be the quotient of $\MCG(S_5)$ such that any homomorphism $B_n\to Q$ has virtually cyclic image, and let $\phi\colon \MCG(S_n)\to Q$ be any homomorphism. Composing the capping homomorphism $B_n\to \MCG(S_{n+1},*)$, the forgetful map $\MCG(S_{n+1},*)\to \MCG(S_n)$, and $\phi$ yields a homomorphism $B_n\to Q$, so $\phi$ must have virtually cyclic image. 

By \cite[Lemma 3.2]{Macpherson_virtcyclic}, an infinite virtually cyclic group surjects onto either $\Z$ or the infinite dihedral group $D_\infty\cong \Z/2*\Z/2$, so it is enough to exclude that $\MCG(S_n)$ surjects onto these. Towards this, recall that the abelianisation of $\MCG(S_n)$ is finite cyclic; therefore $\MCG(S_n)$ does not surject neither onto $\Z$ nor onto $D_\infty$ (which maps onto the non-cyclic Abelian group $\Z/2\times \Z/2$).
\end{proof}

\noindent We will prove both statements of \Cref{thm:ubertheorem} at the same time, highlighting the adjustments needed for the case $n=5$ in parentheses. For the sake of light notation, let $H=B_n$, let $G=\MCG(S_5)$ (resp. $\MCG(S_5,*)$ if $n=5$), and let $\ov X$ be the curve graph of $S_5$. The arguments from \cite[Section 2.3.1]{short_HHG:I} show that $G$ admits a colourable short HHG structure, with support graph $\ov X$ and cyclic directions generated by half Dehn twists. Fix a positive integer $M$, and let 
\begin{equation}\label{eq:def_of_N_and_N'}
    N'=M((n-1)^2)!\mbox{ and }N=N'(5n)!
\end{equation}
We require that $M$ is deep enough that the quotients $G_N$, $G_{N'}$, and $G_M$ all satisfy the properties from \Cref{subsec:prop_of_dehnfill}.

We will argue that $Q=G_N$, which is non-elementary hyperbolic by \Cref{lem:nonelm_hyp_short_quot}, is the desired quotient. Consider a homomorphism $\phi\colon H\to G_N$. We shall consider various cases depending on the image of $\tau$, each case corresponding to one of the lemmas below.

    \begin{lem}\label{claim:finite_order_hdt}
        If $\phi(\tau)$ has infinite order then $\phi(H)$ is virtually cyclic.
    \end{lem}

    \begin{proof}
        Let $E\le G_N$ be the maximal virtually cyclic subgroup containing $\phi(\tau)$, which exists by e.g. \cite[Lemma 6.5]{DGO}. $E$ contains the centraliser of $\phi(\tau)$ by \cite[Corollary 6.6]{DGO}; therefore $\phi(\tau_i)\in E$ for all $i=3,\ldots, n-1$, since $\gamma_1$ and $\gamma_i$ are disjoint for $i$ in this range. Notice also that $\phi(\tau')$ is conjugated to $\phi(\tau)$, so it must have infinite order as well; in particular, $E$ is also the maximal virtually cyclic subgroup containing $\phi(\tau')$, and therefore its centraliser. Finally, since there are at least $5$ punctures, $\tau'=\tau_{n-1}$ and $\tau_2$ commute, and the same argument shows that $\phi(\tau_2)\in E$. We thus proved that $\phi$ maps the generating set $\{\tau_1,\ldots, \tau_{n-1}\}$ inside the virtually cyclic group $E$, and therefore $\phi(H)$ is virtually cyclic.
    \end{proof}

    \begin{rem}
        In the argument above we only used that $G_N$ is hyperbolic, so in fact any homomorphism from $H$ to a hyperbolic group has virtually cyclic image provided that $\tau$ has infinite order.
    \end{rem}

    \noindent From now on, we assume that $\phi(\tau)$ has finite order. Towards applying \Cref{cor:hyp_satisfied}, we look at the fixed point set of $\phi(\tau)$ inside $\ov X_N$.

    \begin{lem}\label{lem:no_fix}
        If $\phi(\tau)$ does not fix any simplex of $\ov X_N$ setwise, then $\phi(H)$ is finite.
    \end{lem}
    \begin{proof}
        For every $1\le i<j\le n-1$ such that $j-1>1$, let $F_{ij}\coloneq \langle \phi(\tau_i), \phi(\tau_j)\rangle$, which is a finite Abelian subgroup of $G_N$. Since $\phi(\tau)$, hence any of its conjugate, does not preserve any simplex of $\ov X_N$, \Cref{cor:hyp_satisfied} yields that each $F_{ij}$ is the isomorphic image under $\phi_N$ of a finite subgroup $F'_{ij}\le G$, which must therefore be Abelian. By \cite[Theorem 3]{Stukow_conj_Class_finite_sgr}, each $F'_{ij}$ is isomorphic to a subgroup of either $\Z/4$ or one of the Dihedral groups $D_3\cong \Z/3 \rtimes \Z/2$ and $D_5\cong \Z/5 \rtimes \Z/2$. Notice that Abelian subgroups of these three finite groups are cyclic, so $F_{ij}=\langle\phi(\tau_i)\rangle=\langle \phi(\tau_j)\rangle$ (here we are using that $\phi(\tau_i)$ and $\phi(\tau_j)$ have the same order as they are conjugate). Hence $\phi(H)=\langle F_{ij}\rangle_{1\le i<j\le n-1}=\langle \phi(\tau)\rangle$ is finite, as required. 
    \end{proof}

    \noindent In view of \Cref{lem:no_fix}, we henceforth assume that $\phi(\tau)$ is an element of the setwise stabiliser of a simplex of $\ov X$ setwise. We now describe such an element:
    \begin{lem}\label{lem:simplexstab}
        Let $g\in G_N$ be a finite order element.
        \begin{enumerate}
            \item\label{simplexstab_1} If $g$ fixes an edge $\{v_N,w_N\}$ of $\ov X_N$ setwise, then $g^2\in \langle Z_{v_N}, Z_{w_N}\rangle$.
            \item\label{simplexstab_2} If $g$ fixes a vertex $v_N$ of $\ov X_N$, but does not preserve any edge setwise, then $g^q\in Z_{v_N}$ for some $2\le q\le 3$; furthermore, $g$ fixes no vertex of $\ov X_N$ other than $v_N$.
        \end{enumerate}
    \end{lem}

    \begin{proof}
        \eqref{simplexstab_1} As a consequence of \cite[Corollary 3.18]{short_HHG:II}, the setwise stabiliser of $\{v_N,w_N\}$ is the $\pi_N$-image of the setwise stabiliser in $G$ of any edge $\{v,w\}\subset \ov X$ lifting $\{v_N,w_N\}$. This is generated by the half Dehn twists around $v$ and $w$ and by any involution swapping $v$ and $w$, so $g^2$ is the image of a multitwist.

        \eqref{simplexstab_2} Let $v\in \ov X^{(0)}$ be a representative for $v_N$. Since $g$ does not fix any $w_N\in \link{\ov X_N}{v_N}$, \Cref{cor:torsion_of_hv} yields that the image of $g$ inside $H_{v_N}$ comes from a finite-order element $r$ of $H_v$. By inspection of \cite[Section 7.1.1]{FM_primer}, $r$ is either a rotation of angle $\pi$ fixing two punctures, or a rotation of angle $2\pi/3$ fixing one puncture. In turn, depending on the rotation angle, $r$ admits a lift $R\in \Stab{G}{v}$ which is a square root of the Dehn twist $T$ around $v$, or a cubic root of $T$ or $T^2$, respectively. Thus $g=\pi_N(h^s R)$, where $h$ is the half Dehn twist around $v$, whose square is $T$, and $s\in \Z$. This immediately shows that $g^q\in Z_{v_N}$, where $2\le q\le 3$ is the order of $r$. For later purposes we also notice that, since $g=\pi_N(h^s R)$ and $R$ is a root of a small power of $T$, some power of $g^2$ acts on $\C{\Delta_{v_N}}$ with minimum displacement greater than the constant $\widetilde{E}$ from \Cref{lem:mindispl_and_commutation}.

        We now prove that $v_{N}$ is the only fixed point of $g$ in $\ov X_N$. Indeed, by assumption $g$ does not fix any $w_N\in \link{\ov X_N}{v_N}$. Moreover, if $g$ fixes some $u_N\not\in\CStar{v_N}$, then $gZ_{u_N}g^{-1}=Z_{g(u_N)}=Z_{u_N}$; hence the square power of $g$ commutes with $Z_{u_N}$, contradicting \Cref{lem:mindispl_and_commutation}.
    \end{proof}
    
    \noindent We momentarily leave the case where $\phi(\tau)$ lies in a cyclic direction for later, and deal with the others in \Cref{lem:fix_edge} and \Cref{lem:fix_vertex} below.
    \begin{lem}\label{lem:fix_edge}
        Suppose that $\phi(\tau)$ fixes an edge $\{v_N,w_N\}$ of $\ov X_N$ setwise, but does not belong to either $Z_{v_N}$ or $Z_{w_N}$. Then $\phi(H)$ is finite.
    \end{lem}

\begin{proof}
    By inspection of \Cref{lem:simplexstab}.\eqref{simplexstab_1}, the square $\phi(\tau)^2$ belongs to $\langle Z_{v_N},Z_{w_N}\rangle$. If $\phi(\tau^2)$ is trivial then all its conjugates are also trivial. In this case $\phi$ vanishes on the pure braid group, which has finite index in $H$, and therefore $\phi(H)$ is finite.

    Suppose now that $\phi(\tau^2)=z_{v_N}^r z_{w_N}^s$ is non-trivial, and notice that both $r$ and $s$ must be non-zero since we are assuming that $\phi(\tau)\not \in Z_{v_N}\cup Z_{w_N}$. Let $\sigma$ be the half Dehn twist around any of the curves $\eta^\pm_{i,j}$ from \Cref{fig:gen_of_bn}  for $3\le i<j\le n $; the choice of $i$ and $j$ in this range ensures that $\sigma$ and $\tau$ commute. Since $\tau$ and $\sigma$ are conjugate in $H$, there exists $g\in G_N$ such that 
    $$\phi(\sigma)^2=g\phi(\tau)^2g^{-1}=z_{g(v_N)}^r z_{g(w_N)}^s.$$
    By \Cref{cor:commuting_multitwist}.\eqref{item_comm_mult2}, $g$ must permute $\{v_N,w_N\}$, so $\phi(\sigma)^2\in \langle Z_{v_N}, Z_{w_N}\rangle$. Now let $\sigma'$ be the half Dehn twist around any of the curves $\eta^\pm_{i',j'}$ from \Cref{fig:gen_of_bn}, for $1\le i'\le 2$ and $i'<j'\le n $. Since $n\ge 5$, $\sigma'$ commutes with some curve $\sigma$ as above, so the same argument yields that 
    $\phi(\sigma')^2\in \langle Z_{v_N}, Z_{w_N}\rangle$. Summarising, we proved that $\phi$ maps the Dehn twists on the curves from \Cref{fig:gen_of_bn}, which generate the pure braid group, inside the finite subgroup $\langle Z_{v_N}, Z_{w_N}\rangle$, and again this implies that $\phi(H)$ is finite.
\end{proof}

\begin{lem}\label{lem:fix_vertex} 
    Suppose that $\phi(\tau)$ fixes a vertex $v_N$ of $\ov X_N$, but does not preserve any edge setwise. Then $\phi(H)$ is finite.
\end{lem}

\begin{proof} The proof in this setting is a mixture of the arguments from \Cref{lem:no_fix} and \Cref{lem:fix_edge}, as we now explain. 

Firstly, \Cref{lem:simplexstab}.\eqref{simplexstab_2} gives that $v_N$ is the only vertex of $\ov X_N$ fixed by $\phi(\tau)$. Now, for all $i=3,\ldots, n-1$, $\phi(\tau_{i})$ is a conjugate of $\phi(\tau)$, so it must fix a single vertex of $\ov X_N$ as well. Moreover, since it commutes with $\phi(\tau)$, it acts on its fixed set in $\ov X_N$, and therefore the unique fixed point of $\phi(\tau_i)$ is $v_N$. Hence we can look at the images of $\phi(\tau)$ and $\phi(\tau_i)$ inside the hyperbolic group $H_{v_N}$, call them $y$ and $y_i$, which are conjugate and commute.

We now argue as in \Cref{lem:no_fix}. The subgroup $\langle y,y_i\rangle$ does not fix any vertex $w_N\in \link{\ov X_N}{v_N}$, so by \Cref{cor:torsion_of_hv} it must be the projection of a finite Abelian subgroup $F\le H_v$. In turn, every finite-order element in $H_v$ permutes at least two punctures, so $F$ injects inside the symmetric group $S_3$, and all Abelian subgroups of the latter are cyclic. This shows that $y_i\in \langle y\rangle$, and therefore $\phi(\tau_i)$ belongs to the finite group generated by $z_v$ and any preimage $x$ of $y$ in $\Stab{G_N}{v_N}$. We can now apply the same argument to the commuting pair $\phi(\tau_2)$ and $\phi(\tau_{n-1})$ to get that $\phi(H)=\langle \phi(\tau_i)\rangle_{1\le i\le n-1}$ is contained inside $\langle z_v, x\rangle$, and the latter is finite as it is a finite cyclic extension of a finite cyclic group.
\end{proof}

\noindent The following Lemma is the last case of  \Cref{thm:ubertheorem}.
\begin{lem}\label{lem:phi_is_dehn}
    Suppose that $\phi(\tau)\in Z_{v_N}$, for some vertex $v_N$ of $\ov X_N$. Then $\phi(H)$ is virtually cyclic.
\end{lem}

\begin{proof}
This is the most involved case. Firstly, we consider what happens if $\phi(T)$ has large order.
    \begin{claim}\label{claim:phiT}
        If $\phi(T)$ has order greater than $5$, then $\phi(H)$ is virtually cyclic.
    \end{claim}
    \begin{claimproof}[Proof of \Cref{claim:phiT}]
        If $\phi(T)$ has infinite order, then $\phi(H)$ lies in the centraliser of $\phi(T)$, which is virtually cyclic as $G_N$ is hyperbolic.

        Now assume the order of $\phi(T)$ is finite and greater than $5$. Again, torsion elements in $\MCG(S_5)$ have order at most $5$, so \Cref{cor:hyp_satisfied} implies that $\phi(T)$ fixes a simplex of $\ov X$ setwise. In turn, by \Cref{lem:simplexstab}, we find some $1\le k\le 3$ such that $\phi(T^k)$ is the image of either a multitwist or a power of a half Dehn twist. In the former situation, $\phi(H)$, which centralises $\phi(T^k)$, must lie in the centraliser of the image of a multitwist, which is finite by \Cref{lem:recognise_(multi)twist}. 
        
        Suppose instead that $\phi(T^k)\in Z_{w_N}$ for some $w_N\in \ov X_N^{(0)}$, so that $\phi(H)\le \Stab{G_N}{w_N}$. Since $\phi(\tau)\in Z_{v_N}$ and $\phi(T^k)\in Z_{w_N}$ commute, the arguments from \Cref{claim:graphiso} show that $w_N\in \CStar{v_N}$. Now recall that, for every $i=1,\ldots, n-1$, $\phi(\tau_i)$ is conjugated to $\phi(\tau)$ by an element $g_i$ inside $\phi(H)\le \Stab{G_N}{w_N}$, so that $\phi(\tau_i)\in Z_{g_iv_N}$. 
        
        If $v_N=w_N$, then $g_iv_N=w_N$ for all $i$, because each $g_i$ fixes $w_N$. This implies that $\phi(H)\le Z_{w_N}$, which is finite. Suppose instead that $v_N\neq w_N$, so that $g_iv_N\in\link{\ov X_N}{w_N} $ for all $i$, and we shall derive a contradiction. Notice that at least two $g_iv_N$ must be distinct, as otherwise $\phi(H)\le Z_{v_N}$ while $\phi(T^k)\in Z_{w_N}$. This also implies that there exist $i,j$ with $|i-j|>1$ and $g_iv_N\neq g_{j}v_N$, for otherwise, using that there are at least 5 indices, we readily see that all $g_iv_N$ would coincide. Since $|i-j|>1$, $\phi(\tau_i)$ and $\phi(\tau_j)$ commute, so $g_iv_N$ and $g_jv_N$, which do not coincide, must be $\ov X_N$-adjacent by \Cref{cor:commuting_multitwist}.\eqref{item_comm_mult1}. However this would violate the fact that $\ov X_N$ has no triangles, as $g_iv_N$ and $g_jv_N$ belong to $\link{\ov X_N}{w_N}$.
    \end{claimproof}

    In view of \Cref{claim:phiT}, we now assume that $\phi$ vanishes on $T^c$ for some $c\le 5$. Now recall that we set $N'=N/(5n)!$ in \Cref{eq:def_of_N_and_N'}. The next Claim shows that the composition $H\xrightarrow[]{\phi} G_N\xrightarrow[]{\pi'}G_{N'}$ has finite image, where $\pi'\colon G_N\to G_{N'}$ is the natural projection:

    \begin{claim}\label{claim:phi_cap_pi'}
        Let $\mathcal K\le G_N$ be the subgroup generated by the image of all $N'$-th powers of half Dehn Twists. Then $\phi(H)\cap \mathcal K$ has index at most $(n-1)^2$ in $\phi(H)$.
    \end{claim}

    \begin{claimproof}[Proof of \Cref{claim:phi_cap_pi'}] Let $x\in H$ be an $n$-th root of $T$ which cyclically permutes all punctures. By \cite[Lemma 6.2]{ChenKordekMargalit}, for every $1\le i\le n-1$ the subgroup normally generated by $x^k$ contains the commutator subgroup; furthermore, since $x$ is an $n$th root of $T$, it maps to $n-1$ inside the abelianisation of $H$, so the index of $\langle\langle x^k\rangle\rangle$ inside $H$ is at most $(n-1)^2$.

    Now let $t=\phi(x)$, which has order at most $5n$ by \Cref{claim:phiT}. If $t^k$ is trivial for some $1\le i\le n-1$ then $\phi$ vanishes on the finite-index subgroup $\langle \langle t^k\rangle\rangle$, and therefore $\phi(H)$ is finite. Otherwise, suppose that the order of $t$ is at least $n$, and in particular $t$ is not the image of any finite-order element of $G$ (if $n=5$ we use that a finite-order element of $\MCG(S_5,*)$ has order at most $4$). Hence $t$ preserves a simplex of $\ov X_N$ by \Cref{cor:hyp_satisfied}, and by \Cref{lem:simplexstab} there exist $1\le q\le 3$ and an edge $\{u_N,w_N\}$ of $\ov X_N$ such that $t^q\in \langle Z_{u_N},Z_{w_N}\rangle$.

    
    Since $t$ has order at most $5n$, while $Z_{u_N}\cong Z_{w_N}\cong \Z/N\Z$ by \Cref{lem:order_of_dt_in_shortquot}, we must have that $t^q\in \langle N'Z_{u_N}, N'Z_{w_N}\rangle$. This shows that $\mathcal K$ contains the normal closure $\langle\langle t^q\rangle\rangle$, and since $q\le 3<n-1$ the latter has index at most $(n-1)^2$ inside $\phi(H)$.
    \end{claimproof}

    Now, recall that we are working under the assumption that $\phi(\tau)\in Z_{v_N}$, and we want to show that $\phi(H)$ is virtually cyclic. By \Cref{claim:phi_cap_pi'}, $\phi(\tau)^{((n-1)^2)!}\in \mathcal K\cap Z_{v_N}$, and the latter coincides with $N'Z_{v_N}$ by \Cref{lem:order_of_dt_in_shortquot}. Hence $\phi(\tau)\in M Z_{v_N}$, where $M$ was defined as $N'/((n-1)^2)!$ in \Cref{eq:def_of_N_and_N'}. Since all generators of $H$ are conjugate, we therefore get that $\phi(H)\le \langle MZ_{v_N}\rangle_{v_N\in \ov X_N^{(0)}}$. 

    From here, we will show that indeed $\phi(H)\in MZ_{v_N}$, concluding the proof. Towards this, the arguments from \cite[Section 3]{short_HHG:II} describe $\langle MZ_{v_N}\rangle_{v_N\in \ov X_N^{(0)}}$ as the subgroup generated by a \emph{composite rotating family}, a notion introduced by Dahmani in \cite{dahmani:rotating}. We do not need the full power of this machinery: for our purposes, it is enough to know that, for every countable ordinal $\alpha$, there exists a subgroup $N_\alpha\le G_N$ with the following properties:
    \begin{itemize}
        \item If $\alpha=0$ then $N_{\alpha}=MZ_{v_N}$.
        \item If $\alpha=\beta+1$ for some ordinal $\beta$, then $N_{\alpha}$ splits as a graph of groups over $N_{\beta}$;
        \item If $\alpha$ is a limit ordinal, then $N_{\alpha}=\bigcup_{\beta<\alpha} N_\beta$.
        \item $\bigcup_{\alpha} N_\alpha=\langle MZ_{v_N}\rangle_{v_N\in \ov X_N^{(0)}}$.
    \end{itemize}
    Let $\alpha$ be the first ordinal such that $\phi(H)\le N_\alpha$. Notice that $\alpha$ is not a limit ordinal: indeed, if $\phi(H)\le N_{\alpha}=\bigcup_{\beta<\alpha} N_\beta$, then each $\phi(\tau_i)$ belongs to $N_{\beta_i}$ for some $\beta_i<\alpha$, and therefore $\phi(H)\le N_{\max_{i=1,\ldots, n-2} \beta_i}\lneq N_\alpha$, against the definition of $\alpha$. 

    Notice also that $\alpha$ cannot be a successor ordinal either. Indeed, if there is $\beta$ such that $\alpha=\beta+1$, then $\phi(H)$ is contained inside $N_{\beta+1}$ but not $N_{\beta}$, and therefore splits as a graph of groups over $\phi(H)\cap N_{\beta}$. However, $\phi(H)$ is a quotient of $H$ with finite abelianisation, since we are assuming that $\phi$ vanishes on a power of $T$; hence $\phi(H)$ has Serre's property (FA), which forbids any non-trivial splitting (see e.g. \cite[page 6]{FA_MCG} and the arguments therein).

    We must therefore have that $\alpha=0$, so that $\phi(H)\le MZ_{v_N}$, as promised. This concludes the proof of \Cref{lem:phi_is_dehn}, and in turn that of \Cref{thm:ubertheorem}.
\end{proof}

\bibliography{biblio.bib}
\bibliographystyle{alpha}

\end{document}